\documentclass[11pt]{amsart}
\usepackage{color}
\usepackage[capitalize,nameinlink,noabbrev,nosort]{cleveref}
\usepackage{graphicx}
\usepackage{tikz}
\newcommand{\tikzcircle}[2][black, fill=blue, opacity=.2]{\tikz[baseline=-0.5ex]\draw[#1,radius=#2] (0,0) circle ;}%

\newtheorem{thm}{Theorem}[section]
\newtheorem{lemma}[thm]{Lemma}

\newtheorem{cor}[thm]{Corollary}
\newtheorem{defn}[thm]{Definition}

\newtheorem{conj}[thm]{Conjecture}
\newtheorem{problem}[thm]{Problem}
\newtheorem{example}[thm]{Example}
\newtheorem{remark}[thm]{Remark}

\newtheoremstyle{definition}
{3pt} 
{3pt} 
{} 
{} 
{\bfseries} 
{.} 
{.5em} 
{} 

\newcommand{\Comment}[1]{{\color{blue} \sf ($\clubsuit$ #1 $\clubsuit$)}}
\newcommand{\CommentO}[1]{{\color{red} \sf ($\clubsuit$ #1 $\clubsuit$)}}

\DeclareMathOperator{\wt}{wt}
\DeclareMathOperator{\Tab}{Tab}

\DeclareMathOperator{\row}{row}

\DeclareMathOperator{\Mat}{Mat}
\DeclareMathOperator{\hMat}{{Mat}}

\DeclareMathOperator{\tr}{tr}
\DeclareMathOperator{\arr}{arr}

\DeclareMathOperator{\States}{States}
\DeclareMathOperator{\2ASEP}{ASEP}

\DeclareMathOperator{\Pos}{Pos}
\DeclareMathOperator{\dis}{dis}

\DeclareMathOperator{\dist}{dis}
\DeclareMathOperator{\west}{row}

\DeclareMathOperator{\Sym}{Sym}
\DeclareMathOperator{\EA}{\Pos_{12}}
\DeclareMathOperator{\EE}{\Pos_2}

\DeclareMathOperator{\CRT}{CRT}

\DeclareMathOperator{\cyc}{cyc}
\DeclareMathOperator{\rec}{rec}

\newcommand{\Z}{\mathbb Z}
\newcommand{\tsigma}{\tilde{\sigma}}

\title[Cylindric rhombic tableaux and the ASEP on a ring]{Cylindric rhombic tableaux and the two-species ASEP on a ring}
\date{\today}
\author{Sylvie Corteel}
\address{Laboratoire d'Informatique Algorithmique: Fondements et Applications,
Centre National de la Recherche Scientifique et Universit\'e Paris Diderot,
Paris 7, Case 7014, 75205 Paris Cedex 13
France}
\email{corteel@liafa.univ-paris-diderot.fr}
\author{Olya Mandelshtam}
\address{Department of Mathematics,
Brown University, Providence, RI}
\email{olya@math.brown.edu}
\author{Lauren Williams}
\address{Department of Mathematics,
Harvard University}
\email{williams@math.harvard.edu}
\thanks{SC was partially funded by the ``Combinatoire \`a Paris" projet
Emergences 2013--2017 and by ``ALEA Sorbonne" projet IDEX USPC. OM was partially supported by NSF grant DMS-1704874. LW was partially supported by NSF grant DMS-1600447.}

\begin{document}


\begin{abstract}
The asymmetric simple  exclusion process (ASEP) is a model of 
particles hopping on a one-dimensional lattice of $n$ sites. It was introduced
around 1970 \cite{bio, Spitzer}, and since then has been extensively studied
 by researchers in statistical mechanics, probability, and combinatorics.
Recently the ASEP on a lattice with open boundaries has been linked to 
Koornwinder polynomials \cite{CW-Koornwinder, Cantini}, and the ASEP on a ring 
has been linked to Macdonald polynomials \cite{CGW}.
In this article we study 
the two-species asymmetric simple exclusion process (ASEP) on a ring,
in which  
two kinds of particles (``heavy'' and ``light''), as well as ``holes,'' can 
	hop both clockwise  and counterclockwise  (at rates $1$ or $t$
	depending on the particle types)
on a ring of $n$ sites. 
We introduce some new tableaux on a cylinder called
\emph{cylindric rhombic tableaux} (CRT), and use them to give 
a formula for the stationary distribution of the two-species ASEP -- each 
probability is expressed as a sum over all CRT of a fixed type.
When $\lambda$ is a partition in $\{0,1,2\}^n$, we then give a formula for the nonsymmetric Macdonald polynomial
$E_{\lambda}$ and the symmetric Macdonald polynomial $P_{\lambda}$ by 
refining our tableaux formulas for the stationary distribution.  
\end{abstract}

\maketitle
\setcounter{tocdepth}{1}
\tableofcontents

\section{Introduction}

Introduced around 1970 \cite{bio, Spitzer},
the asymmetric simple exclusion process (ASEP) 
is a model of interacting 
particles hopping on a one-dimensional lattice of $n$ sites. 
It has been extensively studied 
by researchers in statistical mechanics \cite{DEHP, USW}, 
probability \cite{Liggett, Liggett2, FerrariMartin, BorodinCorwin}, 
and combinatorics 
\cite{jumping, Angel, BE, CW1, CW-Duke1, CMW}. 
Recently the ASEP on a lattice with open boundaries has been linked to
Koornwinder polynomials \cite{CW-Koornwinder, Cantini}, and the ASEP on a ring
has been linked to Macdonald polynomials \cite{CGW}.  
In particular, it was shown in 
\cite{CGW} that when $q=1$ and $x_i=1$ for all $i$, the Macdonald polynomial $P_{\lambda}$ is the 
\emph{partition function} for the multispecies ASEP on a ring.

In this article we study
the two-species asymmetric simple exclusion process (ASEP) on a ring,
in which
two kinds of particles (``heavy'' and ``light'') hop on a lattice of $n$
sites arranged in a ring.  Two adjacent particles, or a particle and a hole, 
can switch places at a rate $t$ or $1$, depending on their relative weights.
We introduce some new tableaux on a cylinder called
\emph{cylindric rhombic tableaux} (CRT), and use them to give
a formula for the stationary distribution of the ASEP -- each
probability is expressed as a sum over the weights of all CRT of a fixed type, where the weight of each CRT is a series.
When $\lambda$ is a partition in $\{0,1,2\}^n$, we then give a formula for the nonsymmetric Macdonald polynomial
$E_{\lambda}$ and the symmetric Macdonald polynomial $P_{\lambda}$ by
refining our tableaux formulas for the stationary distribution.

When $t=0$, the asymmetric simple exclusion process is called the \emph{totally asymmetric simple exclusion 
process} or TASEP.  
Ferrari and Martin \cite{FerrariMartin} studied the multispecies ($k$-species) TASEP on a ring 
and gave combinatorial formulas for the stationary distribution
 in terms of \emph{multiline queues}; they viewed the $k$-TASEP
on a ring as a projection of 
a Markov process on multiline queues, which can be viewed as a coupled system of $k$ single species TASEPs. 
This work was recently generalized by Martin \cite{Martin} to the case of ASEP (i.e. $t$ is general).
Matrix product formulas were found for the probabilities of the TASEP using probabilistic methods in \cite{EvansFerrariMallick08} and generalized to the ASEP case in \cite{ProlhacEvansMallick09} with an explicit construction in \cite{AritaAyyerMallickProlhac12}. From the statistical mechanics side, other formulas for the $k$-TASEP were found by interpreting the Ferrari-Martin process as a combinatorial $R$ matrix in \cite{KMO15}. 
The inhomogeneous multispecies TASEP was also studied in \cite{AyyerLinusson14}, with a graphical construction that generalized the Ferrari-Martin algorithm for the 2-TASEP, and a general conjecture for the $k$-TASEP which was proved using a generalized Matrix ansatz in \cite{AritaMallick12}. 

 Multiline queues have been used to study many aspects of the ASEP \cite{AyyerLinusson14, AasLinusson18}; in this case we give a bijection between
CRT and multiline queues, which is related to recent work by the second author \cite{Man17}. Also note that Haglund-Haiman-Loehr have a tableaux formula for both symmetric \cite{HHL1} and nonsymmetric Macdonald polynomials \cite{HHL2} using \emph{nonattacking fillings}; we explain the 
relation between nonattacking fillings and multiline queues in \cite{CMW-MLQ}
(they are in bijection when the partition has distinct parts; but in general there are more nonattacking
fillings than multiline queues).

\begin{remark}
	In some sense the results of this paper are subsumed by the results of \cite{CMW-MLQ}, 
	in that the latter has combinatorial formulas that work for Macdonald polynomials
	associated to arbitrary partitions (not just $\lambda \in \{0,1,2\}^n$).  However, 
	since these cylindric rhombic tableaux are significantly different than multiline queues,
	and our methods of proof use the Matrix Ansatz rather than the Hecke algebra,
	we thought that this paper might be of independent interest.
\end{remark}

\section{The ASEP on a ring}

We now define the two-species asymmetric simple exclusion process (ASEP) on a ring.

\begin{defn}\label{2-ASEP}
Let $k$, $r$, and $\ell$ be nonnegative integers which sum to $n$,
	and let $t$ be a constant such that $0 \leq t \leq 1$.
	Let $\States(k,r,\ell)$ 
	be the set of all  words of length $n$ in $\{0,1,2\}^n$
	consisting of $k$ $0$'s, $r$ $1$'s, and $\ell$ $2$'s. 
We consider indices modulo $n$; i.e. 
if $\mu=\mu_1\ldots \mu_n \in \{0,1,2\}^n$, then $\mu_{n+1}=\mu_1$.
The \emph{two-species asymmetric simple exclusion process} 
	$\2ASEP(k,r,\ell)$ on a ring
	is the Markov chain on $\States(k,r,\ell)$
	with transition probabilities $P_{\mu,\nu}$ between states $\mu,\nu\in\States(k,r,\ell)$:
\begin{itemize}
\item If $\mu = A i j B$ and $\nu = A j i B$, where 
$A$ and $B$ are words in $\{0,1,2\}^*$ and $i > j$ are letters
in $\{0,1,2\}$, then 
$P_{\mu,\nu} = \frac{t}{n}$ and $P_{\nu,\mu} = \frac{1}{n}$.
\item Otherwise, $P_{\mu,\nu} = 0$ for $\nu \neq \mu$ and 
$P_{\mu,\mu} = 1-\sum_{\mu \neq \nu} P_{\mu,\nu}$.
\end{itemize}

	\begin{figure}[!ht]
  \centerline{\includegraphics[height=1.5in]{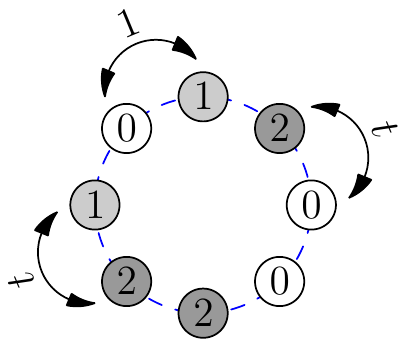}}
\centering
 \caption{The two-species ASEP on a lattice with $8$ sites. There are 
		three holes ($0$'s), two light particles ($1$'s), and 
		three heavy particles ($2$'s), 
		 so we refer
		to this Markov chain as $\2ASEP(3,2,3)$.}\label{parameters}
 \end{figure}

We think of the $1$'s and $2$'s as representing two types of 
particles (``light'' and ``heavy'') which can occupy the sites;
each $0$ denotes an empty site.



\end{defn}


The following Matrix Ansatz \cite{Ansatz1} (see also \cite{Ansatz2})
is a useful tool for computing 
these probabilities in terms of the trace of a certain matrix product.

\begin{thm}[Matrix Ansatz]\label{ansatz}\cite[Section 8]{Ansatz1}
Suppose that  $A_0$, $A_1$, and $A_2$ are matrices (typically infinite)
that satisfy the following relations:
\begin{equation}\label{eq:MA}
A_0A_2=tA_2A_0+(1-t)(A_0+A_2), \qquad A_0A_1=tA_1A_0+(1-t)A_1, \qquad A_1A_2=tA_2A_1+(1-t)A_1.
\end{equation}
Given $\mu\in\States(k,r,\ell)$, we let $\Mat(\mu)$ denote the product of 
matrices obtained from $\mu$ by substituting $A_0$ for $0$, $A_1$ for $1$,
and $A_2$ for $2$.
Then in the $\2ASEP(k,r,\ell)$, 
the steady state probability $\Pr(\mu)$ of state $\mu$ 
is given by 
\[
\Pr(\mu)=\frac{1}{Z_{k,r,\ell}}\tr(\Mat(\mu)),
\]
where $Z_{k,r,\ell}$ is the \emph{partition function} defined  by $[x^ky^rz^{\ell}]\tr((xA_0+yA_1+zA_2)^{k+r+\ell})$.
\end{thm}

\section{Probabilities for the two-species ASEP using cylindric rhombic tableaux}

In this section we define some new combinatorial objects that we call 
\emph{cylindric rhombic tableaux} (or CRT),
and then in \cref{main_result}
we use them to give combinatorial formulas for the 
steady state probabilities of the ASEP.
The proof of our formulas uses the Matrix Ansatz.
Our combinatorial objects will be fillings of certain diagrams composed of squares
and rhombi.  The squares have two horizontal and two vertical edges, while the rhombi have two 
vertical edges as well as two diagonal edges (of slope $1)$, see 
 \cref{Xstrips}.  
The fact that states of the $\2ASEP$ are words in 
$\{0,1,2\}^*$ is related to the fact that there are three types
of lines making up the sides of a square or rhombus: vertical, diagonal, and horizontal.

\begin{defn}
A \emph{(generalized) row} in a CRT is a connected strip of squares and rhombi
which are adjacent along their vertical edges, 
see the left diagram 
in \cref{Xstrips}.   
A \emph{square column} is a connected strip of squares, 
which are adjacent along their horizontal edges; 
and a \emph{rhombic column} is a 
connected strip of rhombi, which are adjacent along their diagonal edges.
\end{defn}

\begin{defn}
Given $\mu\in \{0,1,2\}^*$, we define $\mu|_{12}$ to be the subword of $\mu$ 
	consisting of 1's and 2's. An 
\emph{$\mu$-strip} is a generalized row composed of adjacent squares and rhombi which is obtained by reading $\mu|_{12}$ and appending a square for each 2 and a rhombus for each 
1 to the left of the row; 
see the left diagram 
in \cref{Xstrips}.  
\end{defn} 

\begin{defn}\label{X-path}
Given $\mu\in \{0,1,2\}^*$, we define the \emph{$\mu$-path} $P(\mu)$ to be
the lattice path consisting of south, southwest, and west steps 
obtained by reading $\mu$ and mapping a 0 to a south step, a 1 to a southwest step, and a 2 to a west step; see the bold path at the right of 
\cref{Xstrips}.  
\end{defn}

\begin{figure}[!ht]
  \centerline{\includegraphics[height=1.5in]{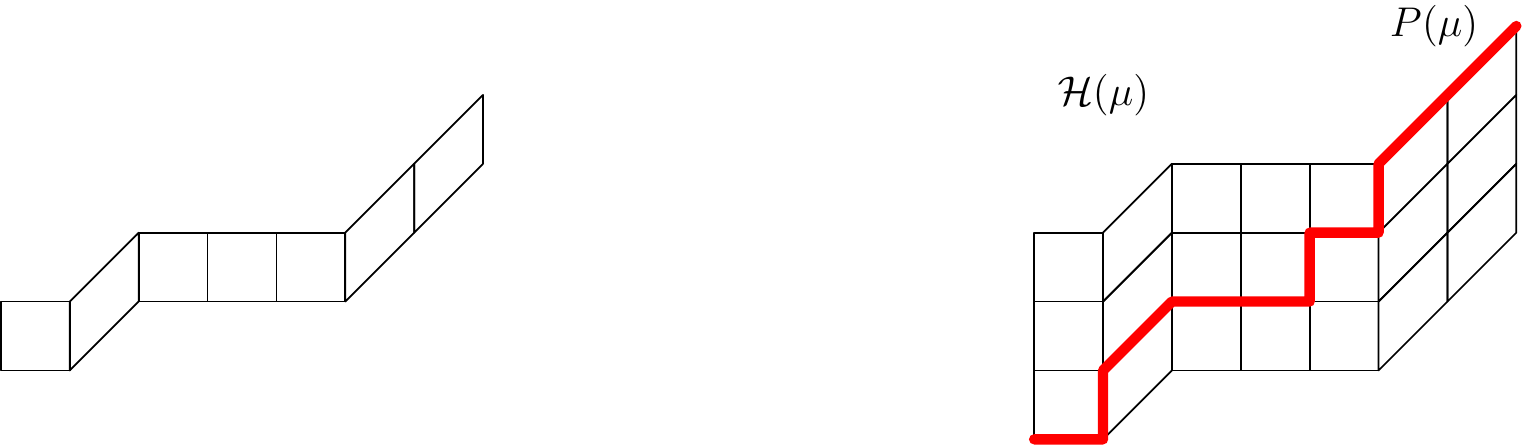}}
\centering
 \caption{For $\mu=1102022102$, the $\mu$-strip is shown at the left, while 
$\mathcal{H}(\mu)$ is shown at the right, with the path $P(\mu)$ superimposed
in bold.}\label{Xstrips}
 \end{figure}

\begin{defn}
Let $\mu\in\States(k,r,\ell)$. Define the \emph{$\mu$-diagram} $\mathcal{H}(\mu)$ to 
be the shape consisting of $k$ $\mu$-strips stacked on top of each other, 
together with the path $P(\mu)$ superimposed onto the shape so that it connects 
the northeast and southwest corners.  
(If $k=0$ then $\mathcal{H}(\mu)$ is defined to be just the path $P(\mu)$.)
See the  diagram at the right in \cref{Xstrips}.
  We  identify the two vertical edges on either end of each 
row; in this way we view the shape on a  cylinder.
Thus the rightmost tile is adjacent to the leftmost tile in each row.
\end{defn}

Note that for $\mu\in\States(k,r,\ell)$, $\mathcal{H}(\mu)$ has  $k$ rows,
$r$ rhombic columns, and 
$\ell$ square columns.
For example in \cref{Xstrips}, $\mathcal{H}(\mu)$
has $3$ rows, $3$ rhombic columns, and $4$ square columns.
For $m\in[k]$, let $\west(m)$ denote the $m$'th row, numbered
 from bottom to top.

\begin{defn}[Cylindric rhombic tableau and arrow ordering]\label{def:CRT}
Choose a word $\mu\in\States(k,r,\ell)$. 
A \emph{cylindric rhombic tableau} (CRT) $T$  of type $\mu$ is a placement of up-arrows into 
the square tiles of the diagram  $\mathcal{H}(\mu)$ so that there is \emph{at most} one up-arrow in each column. 
(We allow columns to be empty.)
We denote the set of cylindric rhombic tableaux
of type $\mu$ by $\CRT(\mu)$.

An \emph{arrow ordering} of $T$ is a labeling of the arrows in each row by 
the numbers $1,\dots,i$, where $i$ is the number of arrows in that row.
	Let $\arr(T)$ denote the total number of arrows in $T$.
We let $\sigma^i$ denote the labeling of the arrows in $\west(i)$, and let
$\{\sigma^i\} = (\sigma^1,\dots,\sigma^k)$.
\end{defn}

For an example, see \cref{fig:CRT_example}.

\begin{defn}\label{def:free}
We say an arrow in tile $s$ is \emph{pointing at} a tile $s'$ if they are in the same column and $s$ is below $s'$ when we read from bottom to top.
We call a square tile \emph{free} if the tile is empty and 
	there is no arrow pointing to it.  (Note that freeness does not depend on the 
	path $P(\mu)$.)
\end{defn}

\begin{figure}[!ht]
  \centerline{\includegraphics[width=2in]{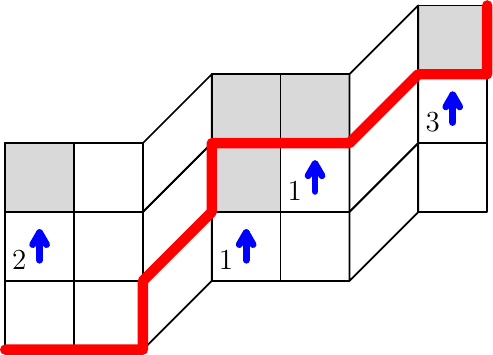}}
\centering
 \caption{A cylindric rhombic tableau $T$ of type $0212201022$ with a chosen arrow ordering $\sigma$.  The
 square tiles that are not free are  grey.}\label{fig:CRT_example}
 \end{figure}

We will define the \emph{weight} of each cylindric rhombic tableau.  To do so,
we need to introduce a few combinatorial statistics.

\begin{defn}
Given a subset $I$ of a finite sequence $U$ where $|U|=m$, we let 
$\Sym_{I,U}$ denote the set of total orders on $I$, which we also call
	\emph{partial permutations}.  We write the elements
of $\Sym_{I,U}$ as strings 
of length $m$, with a $*$ denoting elements not in $I$.  
\end{defn}

For example, if $U=\{1,2,\dots,6\}$ and $I=\{1,2,4\}$, then there are $|I|!$ total orders
on $I$, which we denote by
\[\Sym_{I,U} = \{1\ 2\ *\ 3\ *\ *,\ 1\ 3\ *\ 2\ *\ *,\ 2\ 1\ *\ 3\ *\ *,\ 2\ 3\ *\ 1\ *\ *,\ 3\ 1\ *\ 2\ *\ *,\ 3\ 2\ *\ 1\ *\ *\}.\]

\begin{defn}[Disorder]\label{def:disorder1}
Let $\widetilde{\Sym}_{I,U}$ denote the set of all sequences that can be obtained from the elements of $\Sym_{I,U}$ be inserting a $0$ in an arbitrary position. Given $\tilde{\sigma}\in \widetilde{\Sym}_{I,U}$, we define its \emph{disorder} $\dis(\tilde{\sigma})$ inductively as follows:

  Reading the entries of $\tilde{\sigma}$ from left to right starting
from the $0$,
	we let $\dist_1(\tsigma)$ equal the number of $*$'s or numbers bigger than $1$
we encounter before we reach the $1$.
	We then let $\dist_2(\tsigma)$ be the number of $*$'s or numbers bigger than $2$ we encounter
if we travel from the $1$ to the $2$ from left to right, wrapping around to the beginning
	of $\tsigma$ if necessary.  Similarly,
	$\dist_i(\tsigma)$ is the number of $*$'s or numbers bigger than $i$ we encounter
if we travel from the $i-1$ to the $i$ from left to right, wrapping around 
if necessary.  Finally we define the \emph{disorder} to be 
$\dis(\tsigma) = \dist_1(\tsigma) + \dist_2(\tsigma) + \dots + \dist_{|I|}(\tsigma).$
\end{defn}

If $\tsigma = 0\ 2\ 1\ *\ 3\ *\ *$, then 
$\dist_1(\tsigma) = 1$, $\dist_2(\tsigma) = 4$, $\dist_3(\tsigma) = 1$, and 
$\dis(\tsigma) = 6$.

\begin{remark}\label{betrayal}
In a recent paper \cite{KalizeszewskiMorse17}, a statistic very similar to disorder, 
	called \emph{betrayal}, was introduced on certain colored words in a formula for modified symmetric Macdonald polynomials $\tilde{H}_{\lambda}$. It would be interesting to understand the connection between the statistics on these different objects. 
\end{remark}

\begin{defn}[From an arrow ordering to a partial permutation]\label{def:tsigma}
Given a cylindric rhombic tableau $T$ 
	and an arrow ordering $\{\sigma^i\}$, we associate a partial
	permutation to each row of $T$ as follows.
We fix  $\west(i)$ and read its elements from \textbf{right to left},
skipping over non-free square tiles, but 
	recording free square tiles and rhombic tiles by a $*$, and
	arrows by their label.  We also record the vertical line in 
	$P(\mu)$ by a $0$.
We denote this partial permutation by $\tsigma^i$.
\end{defn}

For example, the rows of the tableau in \cref{fig:CRT_example} would give rise to the sequences
$\tsigma^1 = *\ *\ *\ 1\ *\ 0\ *\ *$, 
$\tsigma^2 = 3\ *\ 1\ 0\ *\ *\ 2$, and 
$\tsigma^3 = 0\ *\ *\ *$ (which are read from left to right).

We now define the \emph{disorder} for arrow orderings of 
 cylindric rhombic tableaux.

\begin{defn}[Disorder of a CRT with an arrow ordering]\label{def:disorder2}
Given a cylindric rhombic tableau $T$ with $k$ rows and an arrow ordering $\{\sigma^i\}=\{\sigma^1,\ldots,\sigma^k\}$, we define the \emph{disorder} 
of $(T,\{\sigma^i\})$ to be 
\[\dis(T,\{\sigma^i\}) = \sum_{i=1}^k \dis(\tsigma^i).\] 
\end{defn}

\begin{example}\label{ex:1}
Using  $(T,\{\sigma^i\})$ from  \cref{fig:CRT_example},
	we compute $\dis(*\ *\ *\ 1\ *\ 0\ *\ *) = 5$, $\dis(3\ *\ 1\ 0\ *\ *\ 2) = 5+2+0=7$, and $\dis(0\ *\ *\ *)=0$, so 
$\dis(T,\{\sigma^i\}) = 12$.
\end{example}

We let  $[i]=[i]_t$ denote the
$t$-analogue of the positive integer $i$, that is,
$[i]=\frac{1-t^i}{1-t} =  1+t+ \ldots+t^{i-1}$.  We also let $[i]!=[1][2]\ldots[i]$.

\begin{defn}\label{tweight_def}
The \emph{$t$-weight} $\wt_t(T)$ of a cylindric rhombic tableau $T$ of type $\mu\in\States(k,r,\ell)$ is computed as follows. 
	
 Given an arrow ordering $\{\sigma^i\}$ of the arrows in $T$, we define
\[
\wt_t(T,\{\sigma^i\})=
t^{\dis(T, \{\sigma^i\})}.
\]
We then define the \emph{$t$-weight} of $T$ to be 
\[\wt_t(T)=
	\frac{[r+\ell-\arr(T)]!}{[r+\ell]!}
\sum_{\{\sigma^i\}} \wt_t(T,\{\sigma^i\}),\]
 where $\{\sigma^i\}$ varies over all possible arrow orderings of $T$.
\end{defn}

\begin{example}
Continuing \cref{ex:1}, 
with 
\cref{fig:CRT_example}, we have  
	$r=2$, $\ell=5$, and 
	$\arr(T) = 4$.
	Thus 
	$\frac{[r+\ell-\arr(T)]!}{[r+\ell]!}=
	\frac{1}{[7] [6] [5] [4]}$.

To compute $\wt_t(T)$, we need to consider all possible arrow orderings.
Note that: 
\begin{itemize}
\item There is only one arrow ordering of $\west(1)$ and of $\west(3)$, so the weight contributed to 
		$\wt_t(T)$ by the possible arrow orderings of $\west(1)$ and $\west(3)$ is just $t^5$.
\item If we represent the arrows versus rhombic/free tiles in $\west(2)$
 by $x$'s and  $*$'s, respectively,
then the content of $\west(2)$ can be encoded by
the sequence $x\ *\ x\ 0\ *\ *\ x$.
		We have 
		$\dis(1\ *\ 2\ 0\ *\ *\ 3)= 6$,
		$\dis(1\ *\ 3\ 0\ *\ *\ 2)= 8$,
		$\dis(2\ *\ 1\ 0\ *\ *\ 3)= 11$,
		$\dis(2\ *\ 3\ 0\ *\ *\ 1)= 3$,
		$\dis(3\ *\ 1\ 0\ *\ *\ 2)= 7$, and
		$\dis(3\ *\ 2\ 0\ *\ *\ 1)= 6$.
		Thus the weight contributed by the possible arrow labelings of $\west(2)$ (only one of which is shown in \cref{fig:CRT_example})
		of is $t^3+2t^6+t^7+t^8+t^{11}$.
\end{itemize}

Letting $I_i$ denote the positions of the arrows in $\west(i)$ and $U_i$
denote the positions of the arrows and free tiles in $\west(i)$, 
we can write 
	the total weight of this tableau for all possible arrow orderings $\{\sigma^i\}=(\sigma^1,\dots,\sigma^5)$  as 
\[
	\wt_t(T)=\frac{1}{[7][6][5][4]}\prod_{m=1}^3\sum_{\sigma^m\in \Sym_{I_m, U_m}} t^{\dis(\tilde{\sigma}^m)}  =\frac{(t^5)(t^3+2t^6+t^7+t^8+t^{11})}{[7][6][5][4]}.
\]
\end{example}

\begin{remark}
It is interesting to note that if $I=U$, disorder on $\Sym_{I,U} = \Sym_{|I|}$ 
is a Mahonian statistic, i.e. it has the same distribution as \emph{inversions}.
\end{remark}

\begin{defn}[Combinatorial partition function]\label{Z1}
Given $\mu\in\States(k,r,\ell)$, we define 
\[\Tab_t(\mu):=\sum_{T\in\CRT(\mu)} \wt_t(T).\] 
We also define
 the \emph{combinatorial partition function} of the cylindric rhombic tableaux to be
	\[\mathcal{Z}_{k,r,\ell}(t) = \sum_{\mu\in\States(k,r,\ell)} \Tab_t(\mu).\] 
\end{defn}

We are finally ready to state the first main result of this paper.

\begin{thm}\label{main_result}
Consider the two-species asymmetric simple exclusion process
 $\2ASEP(k,r,\ell)$. 
Then the steady state probability of being in state $\mu$, where 
$\mu\in\States(k,r,\ell)$, is 
\[
	\Pr(\mu) = \frac{\Tab_t(\mu)}{\mathcal{Z}_{k,r,\ell}(t)}, 
\]
where $\Tab_t(\mu)$ and 
	$\mathcal{Z}_{k,r,\ell}(t)$ are as in \cref{Z1}.
\end{thm}

\begin{example}
To compute the steady state probability 
$\Pr(221100)$ of the state $221100$ of 
the $\2ASEP(2,2,2)$, we need to 
sum the weights of all cylindric rhombic tableaux of type
$221100$, see 
	\cref{EEAADD}.  We then find that 
	$\Pr(221100) = 
\frac{(t+1)(t^4+t^3+6t^2+t+6)}{[4][3]Z_{2,2,2}}.$
\begin{figure}[!ht]
  \centerline{\includegraphics[width=\textwidth]{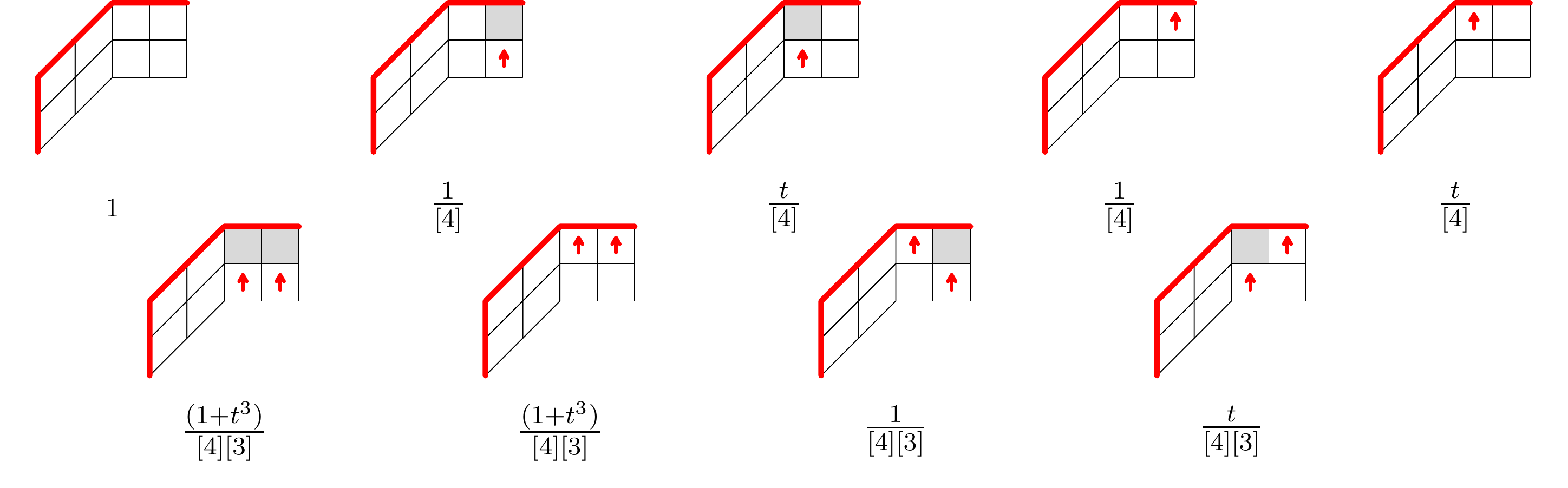}}
\centering
	\caption{
	 The cylindric rhombic tableaux of type $221100$ and their weights.}
\label{EEAADD}
 \end{figure}
\end{example}

\begin{example}
To compute the steady state probability 
$\Pr(201021)$ of the state $201021$ of 
the $\2ASEP(2,2,2)$, we need to 
sum the weights of all cylindric rhombic tableaux of type
$201021$, see 
	\cref{EDADEA}.  We then find that 
	$\Pr(201021) 
	= \frac{(t+1)(t^2+t+1)(2t^2+t+2)}{[4][3]\mathcal{Z}_{2,2,2}(t)}.$
\begin{figure}[!ht]
  \centerline{\includegraphics[width=\textwidth]{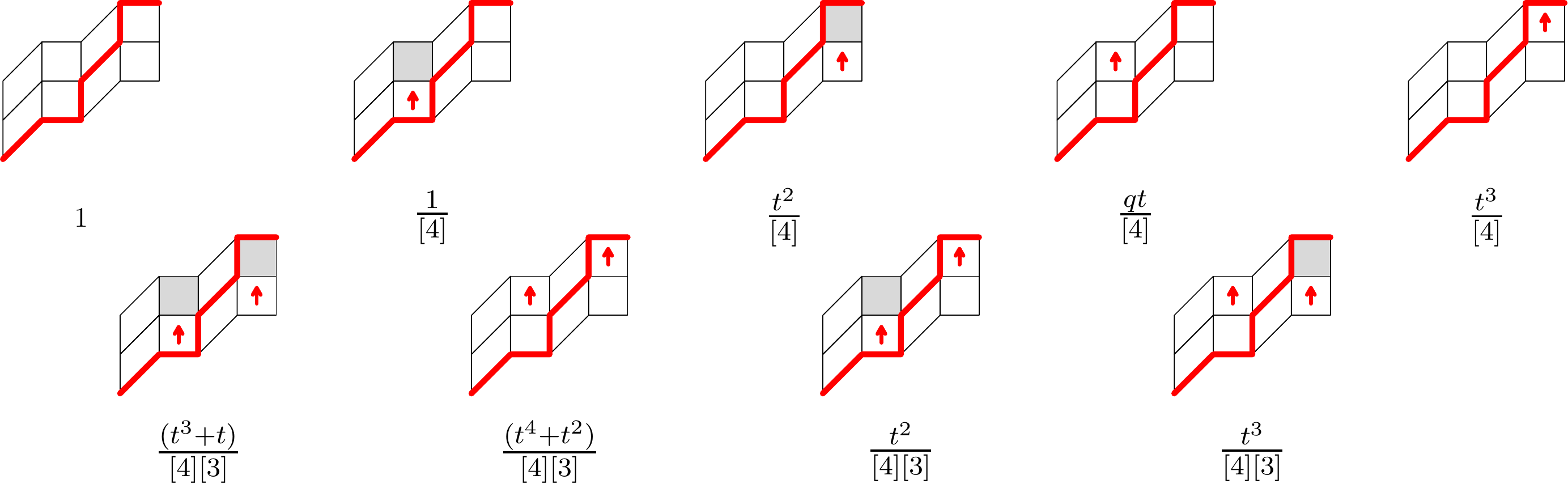}}
\centering
	\caption{
		 The cylindric rhombic tableaux of type $201021$ and their weights.}
\label{EDADEA}
 \end{figure}
\end{example}

\begin{remark}
For a given state $\mu$ of the $\2ASEP(k,r,\ell)$, we can multiply $\Tab_t(\mu)$ by the scalar $[r+\ell]!/[r]!$ to obtain polynomials in $t$ with positive coefficients. The resulting
	polynomials display a ``particle-hole symmetry,''  see \cref{conj:complement} and \cref{conj:conjugate}. 
\end{remark}

\section{Formulas for Macdonald polynomials using cylindric rhombic tableaux}

Symmetric Macdonald polynomials \cite{Macdonald} are a family of multivariable orthogonal polynomials
indexed by partitions, whose coefficients depend on two parameters $q$ and $t$.  
In  recent works \cite{CGW, CGW-arxiv}, 
Cantini, de Gier, and Wheeler gave a link between the multi-species exclusion 
process on a ring and Macdonald polynomials.  
In this section we will give a combinatorial formula for Macdonald polynomials in a special case; the proof of our formula uses some results from \cite{CGW}.

Let $F=\mathbb{Q}(q,t)$ be the field of rational functions in $q$ and $t$, 
and let $m_{\lambda}$
denote the monomial symmetric polynomial indexed by the partition $\lambda$.  The Macdonald polynomials are 
defined as follows.

\begin{defn}
Let $\langle \cdot, \cdot \rangle$ denote the Macdonald inner product on power sum symmetric functions
\cite[Chapter VI, Equation (1.5)]{Macdonald}, where $<$ denotes the dominance order on 
partitions \cite[Chapter I, Section 1]{Macdonald}.  The \emph{Macdonald polynomial} 
$P_{\lambda}(x_1,\dots,x_n; q,t)$ is the unique homogeneous symmetric polynomial in 
$x_1,\dots,x_n$ with coefficients in $F$ which satisfies
\begin{align*}
\langle P_{\lambda}, P_{\mu} \rangle &= 0, \text{ for }\lambda \neq \mu, \\
P_{\lambda}(x_1,\dots,x_n; q,t) &= m_{\lambda}(x_1,\dots,x_n) + \sum_{\mu < \lambda} c_{\lambda,\mu}(q,t)
m_{\mu}(x_1,\dots,x_n),
\end{align*}
i.e. the coefficients $c_{\lambda,\mu}(q,t)$ of the lower degree terms are  determined by 
the orthogonality conditions.
\end{defn}

The \emph{nonsymmetric Macdonald polynomials} $E_{\mu}$, which are indexed by compositions, were later defined by Opdam \cite{Opdam} and Cherednik \cite{Cher1, Cher2} as joint eigenfunctions of a family of commuting operators in the double affine Hecke algebra, with $P_{\lambda}$ obtained as a sum over the $E_{\mu}$:
\[
P_{\lambda} = \sum_{\mu} E_{\mu},
\]
for $\mu$ ranging over all permutations of $\lambda$. For more details,
see \cite{Macdonald}.

In this section we enhance our weight function on tableaux, to include 
 an additional parameter $q$ and variables $x_1,\dots,x_n$
(where $n=k+r+\ell$).  We then give our second main result, which 
is a formula for certain Macdonald polynomials in terms of cylindric 
rhombic tableaux.  
In particular,
we will give a formula for the nonsymmetric Macdonald polynomial
$E_{\lambda}$ 
and a formula
for the symmetric Macdonald polynomial $P_{\lambda}$, 
where $\lambda$ is any partition in $\{0,1,2\}^*$.
Note that Haglund, Haiman and Loehr have given combinatorial formulas 
for both the nonsymmetric Macdonald polynomials and the symmetric Macdonald
polynomials in terms of \emph{nonattacking fillings of composition diagrams} \cite{HHL1, HHL2}; it would
be interesting to understand how our
 formulas relate to theirs.

\begin{defn}\label{tweight}
	We refer to the left and right border of 
	a cylindric tableau (which are identified)
	as its \emph{vertical boundary}.
Given a cylindric rhombic tableau $T$ with path $P(\mu)$
and
arrow ordering $\{\sigma^i\}$, for each row $\row(i)$, we define $\cyc(T,\sigma^i)$ to be the number of times we cross the vertical boundary if we start at the vertical line in $P(\mu)$ in row $i$ and then travel from right to left (wrapping around if necessary) to the arrow labeled $1$, then the arrow labeled $2$, and so on. We define the \emph{cycling} of $(T,\{\sigma^i\})$ to be 
\[\cyc(T,\{\sigma^i\}) = \sum_i \cyc(T,\sigma^i).\]
\end{defn}

Recall that a \emph{recoil} of a (partial) permutation $\sigma$ 
is a pair $(j+1,j)$ such that $\sigma^{-1}(j+1)<\sigma^{-1}(j)$.  In other
words, it is a pair of values $(j+1,j)$ where $j+1$ appears to the 
left of $j$ in $\sigma$.  For example, the partial permutation 
$3*10**2$ has two recoils, $(1,0)$ and $(3,2)$.
Note that $\cyc(T,\sigma)$ for a given row's arrow ordering 
$\{\sigma^i \}$ is equal to the number of recoils of $\tsigma$ 
(see \cref{def:tsigma}).
The cycling statistic defined above will contribute to the power of $q$
associated to each tableau and arrow ordering.

Now given a CRT with path $P(\mu)$, 
let us number the steps of $P(\mu)$ from northeast to southwest using the numbers 
$1,2,\dots,n$, where $n=|\mu|$; see \cref{fig:CRT_example2}.  
This allows us to give every row and column of 
$\mathcal{H}(\mu)$ a unique integer label, and we will subsequently refer to row $i$ and column $j$
using this labeling.

\begin{figure}[!ht]
  \centerline{\includegraphics[width=2in]{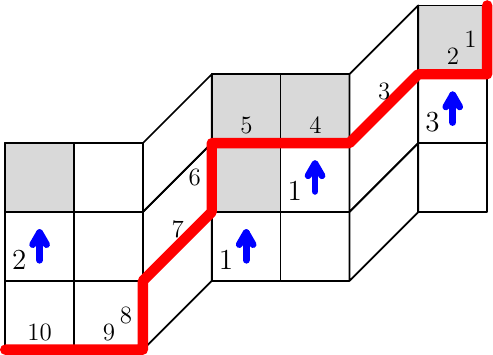}}
\centering
 \caption{A cylindric rhombic tableau  of type $\mu=0212201022$, 
with the steps of $P(\mu)$ labeled from $1$ to $10$.}
\label{fig:CRT_example2}
 \end{figure}

\begin{defn}\label{xweight_def}
The \emph{$x$-weight} $\wt_x(T, \{\sigma^i\})$ of a cylindric rhombic tableau $T$ 
of type $\mu\in\States(k,r,\ell)$ 
with an arrow ordering
$\{\sigma^i\}$  
is computed as follows. 

For each arrow $a$ in $T$, if its label given by the arrow ordering 
is the maximum among all arrows in its row, then we set
$\wt(a) = x_i$, where $i$ is the row label of the square containing $a$.  
Otherwise, we set $\wt(a)=x_j$,
where $j$ is the column label of the square containing $a$. 

For each column $c$ of squares in $T$, if $c$ contains no arrows,
we set $\wt(c)=x_j$, where $j$ is the column label of $c$.  Otherwise we set $\wt(c)=1$.

We also define $\wt_{12}(\mu):=\prod_{k\in \EA(\mu)} x_k$, where 
$\EA(\mu) = \{i \ \vert \ \mu_i = 1 \text{ or }2\}$.

Finally we define the $\emph{$x$-weight}$ 
$\wt_x(T,\{\sigma^i\})$ of $(T,\{\sigma^i\})$ to be 
\[\wt_x(T,\{\sigma^i\}) = \wt_{\EA}(\mu) \prod_a \wt(a) \prod_c \wt(c),\]
where the products are over all arrows $a$ and columns $c$ of squares of $T$.
\end{defn}

\begin{remark}
It follows from the above definition that given 
	a CRT $T$ of type $\mu\in \States(k,r,\ell)$,
	the $x$-weight of $T$ is a monomial in $x_1\dots x_n$
	of degree $r+2\ell$.
\end{remark}

\begin{example}\cref{fig:CRT_example2} shows a cylindric rhombic tableau $T$ of type $0212201022$ together with an 
arrow ordering $\{\sigma^i\}$.  In this example we have 
$\prod_a \wt(a) = x_4 x_6 x_8 x_{10}$, $\prod_c \wt(c) = x_9$,
$\wt_{\EA}(\mu) = x_2 x_3 x_4 x_5 x_7 x_9 x_{10}$.  Therefore
\[\wt_x(T,\{\sigma^i\}) = x_2 x_3 x_4^2 x_5 x_6 x_7 x_8 x_9^2 x_{10}^2.\]
\end{example}

Given a positive integer $i$, let $[i]_{qt} = \frac{1-qt^i}{1-t}$,
and let $[i]_{qt}! = [i]_{qt} [i-1]_{qt} \dots [1]_{qt}$. Note that when $q=1$, $[i]_{qt}$ recovers the quantity $[i]=[i]_t$ we defined earlier. Finally we are ready to define the $qtx$-weight of a cylindric rhombic
tableau.

\begin{defn}\label{weight_def}
Let 
$T$ be a cylindric rhombic tableau  of type $\mu\in\States(k,r,\ell)$, and 
	and let $\{ \sigma^i\}$ be an arrow ordering of its arrows.
	The \emph{$qtx$-weight} $\wt_{qtx}(T,\{\sigma^i\})$ is defined to be
	\begin{equation}
	\wt_{qtx}(T,\{\sigma^i\})=
		{t^{\dis(T, \{\sigma^i\})} q^{\cyc(T,\{\sigma^i\})} 
		\wt_x(T,\{\sigma^i\})}.
	\end{equation}
We then define the \emph{$qtx$-weight} of $T$ to be 
	\[\wt_{qtx}(T)=
	\frac{[r+\ell-\arr(T)]_{qt}!}{[r+\ell]_{qt}!}
	\sum_{\{\sigma^i\}} \wt_{qtx}(T,\{\sigma^i\}),\]
 where $\{\sigma^i\}$ varies over all possible arrow orderings of $T$.
\end{defn}

\begin{defn}
Given $\mu\in\States(k,r,\ell)$, we define 
	\[\Tab_{qtx}(\mu):=\sum_{T\in \CRT(\mu)} \wt_{qtx}(T)\] 
\end{defn}

\begin{example}
	\cref{EDADEAt} shows the cylindric rhombic tableaux of type 
	$201021$.  
	The sum of the weights of all the tableaux 
	is $\Tab_{qtx}(201021)=x_1^2x_3x_5^2x_6+\frac{(x_1+qt^2x_5)x_1x_3x_4x_5x_6}{[4]_{qt}}+\frac{(tx_1+qt^3x_5)x_1x_2x_3x_5x_6}{[4]_{qt}}+\frac{q(t^3x_1+tx_5)x_1x_3x_4x_5x_6}{[4]_{qt}[3]_{qt}}+\frac{q(t^4x_1+t^2x_5)x_1x_2x_3x_5x_6}{[4]_{qt}[3]_{qt}}+\frac{q(t^2+t^3)x_1x_2x_3x_4x_5x_6}{[4]_{qt}[3]_{qt}}$.
\begin{figure}[!ht]
  \centerline{\includegraphics[width=\textwidth]{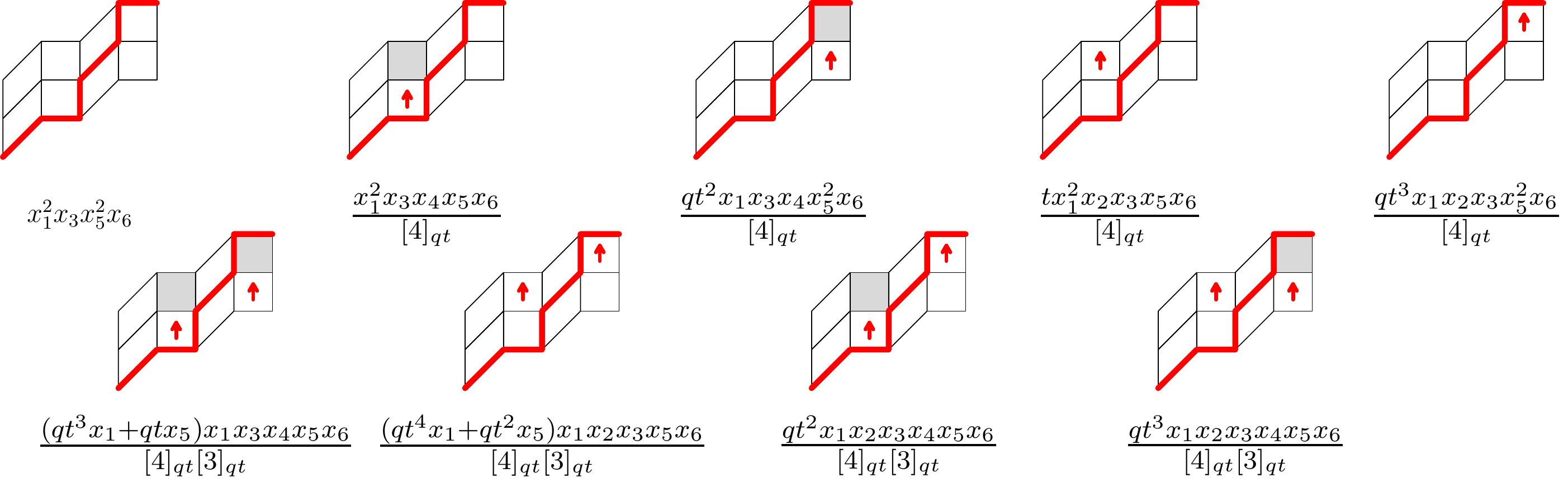}}
\centering
	\caption{
		All cylindric rhombic tableaux of type $201021$, along with their weights.}
	
\label{EDADEAt}
 \end{figure}
\end{example}

The second main result of this paper is the following.

\begin{thm}\label{main_result2}
For any partition $\lambda = (\lambda_1,\dots,\lambda_n)$
 of the form 
$2^\ell 1^r 0^k$, we have that 
the 
nonsymmetric Macdonald polynomial $E_{\lambda}$ is given by 
\begin{equation}
E_{\lambda}(x_1,\dots,x_n; q,t) = \Tab_{qtx}(2^\ell 1^r 0^k).
\end{equation}
Moreover 
the symmetric Macdonald polynomial $\mathcal{P}_{\lambda}$ is given by  
\begin{equation}
\mathcal{P}_{\lambda}(x_1,\dots,x_n; q, t) = \sum_{\mu} \Tab_{qtx}(\mu),
\end{equation}
where the sum runs through all distinct permutations $\mu$ of $\lambda$.
\end{thm}

\begin{example} 
	Using SageMath \cite{sagemath}, 
	we find that the nonsymmetric Macdonald polynomial 
	$E_{221100} = E_{221100}(x_1,\dots,x_6; q,t)$ equals 
 $$E_{221100} = x_1^2x_2^2x_3x_4 + 
\frac{q(x_1+x_2)(x_5+x_6)x_1 x_2 x_3 x_4}{[3]_{qt}} + 
\frac{q^2(1+t)x_1 x_2 x_3 x_4 x_5 x_6}{[3]_{qt}[4]_{qt}}.$$ This agrees with the sum of the weights of 
	the tableaux of type $\mu=221100$, see  \cref{EEAADD2}.

\begin{figure}[!ht]
  \centerline{\includegraphics[width=\textwidth]{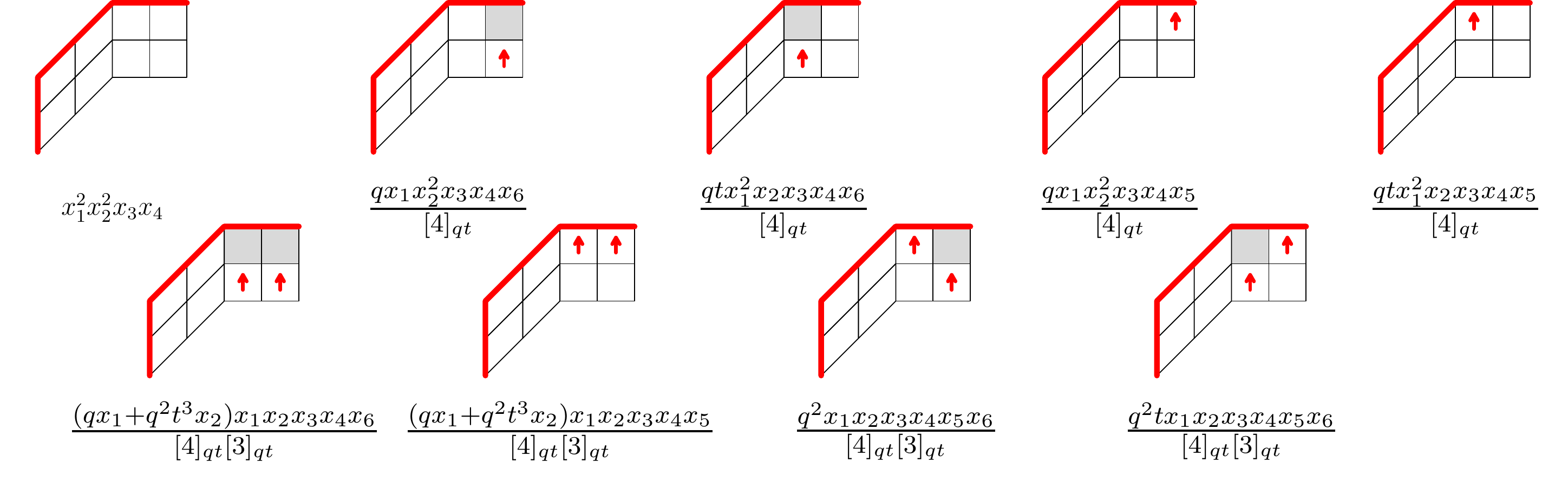}}
\centering
	\caption{All cylindric rhombic tableaux of type $221100$, along with their weights. }
\label{EEAADD2}
 \end{figure}
\end{example}

\section{The Matrix Ansatz and the results of Cantini-deGier-Wheeler}\label{CGW}
In order to prove \cref{main_result} and \cref{main_result2}, we need to
 introduce some matrices from \cite{CGW}, which 
can be used to compute certain  Macdonald polynomials.

\begin{defn}\cite[(53)]{CGW}\label{defn:matrices}
We define semi-infinite matrices $A_0(x)$, $A_1(x)$, $A_2(x)$, and $S$, whose rows and columns are 
indexed by $\Z_{\geq 0}$.  

  Let $A_0(x)=(A_0(x)_{i,j})$ be 
  defined by 
  \[
  A_0(x)_{i,j}=\left\{ \begin{array}{ll}
	  1 & {\rm if}\ i=j\\
	  x & {\rm if}\ i=j-1\\
	  0 & {\rm otherwise}.
  \end{array} \right.
  \]

  Let $A_2(x)=(A_2(x)_{i,j})$ be defined by 
  \[
  A_2(x)_{i,j}=\left\{ \begin{array}{ll}
	  x^2 & {\rm if}\ i=j\\
	  x(1-t^i) & {\rm if}\ i=j+1\\
	  0 & {\rm otherwise.}
  \end{array}\right.
  \]

  Let $A_1(x)=(A_1(x)_{i,j})$ be a {\bf diagonal} matrix defined by
  \[
  A_1(x)_{i,i}=xt^i.
  \]

  Let $S=(S_{i,j})$ be a {\bf diagonal}  matrix defined by 
  \[
  S_{i,i}=q^i.
  \]
\end{defn}

Cantini, deGier, and Wheeler \cite{CGW} 
proved that Macdonald polynomials can be computed in terms of 
the matrices above as follows.  (We restrict to the setting where compositions have parts equal to 
$0$, $1$, or $2$.)

\begin{thm}\cite[(16),(24), Lemma 3]{CGW}\label{CGWTheorem}
Given a composition $\mu = (\mu_1,\dots,\mu_n) \in \{0,1,2\}^n$, let 
$\lambda$ be the partition obtained from $\mu$ by sorting its parts, and 
set \begin{equation}
	\Omega_{\lambda}(q,t) = \prod_{1 \leq i < j \leq s} ({1-q^{j-i} t^{\lambda'_i - \lambda'_j}}),
\end{equation}
where $s$ is the largest part of $\lambda$.
We define
	\begin{equation}\label{eq:f}
f_{\mu}(x_1,\dots,x_n; q,t) = \Omega_{\lambda} \tr [A_{\mu_1}(x_1) \dots A_{\mu_n}(x_n) S].
\end{equation}

For any partition  $\lambda = (\lambda_1,\dots,\lambda_n)\in \{0,1,2\}^*$, 
the 
nonsymmetric Macdonald polynomial $E_{\lambda}$ is given by 
	\begin{equation}\label{eq:nonsymmetric}
		E_{\lambda}(x_1,\dots,x_n; q,t) = f_{\lambda}(x_1,\dots,x_n; q,t).\footnote{Note that \eqref{eq:nonsymmetric} does not hold if we replace 
		$\lambda$ with an arbitrary composition.  Instead the two 
		families of polynomials (the $E$'s and the $f$'s) are related via 
		a triangular change of basis, see \cite[(23)]{CGW}.}
\end{equation}
Moreover the symmetric Macdonald polynomial $\mathcal{P}_{\lambda}$ is given by  
\begin{equation}
\mathcal{P}_{\lambda}(x_1,\dots,x_n; q, t) = \sum_{\mu} f_{\mu}(x_1,\dots,x_n; q, t),
\end{equation}
	where the sum runs through all distinct permutations $\mu$ of $\lambda$.\footnote{Note that \cite{CGW} uses some unusual conventions for Macdonald polynomials.  In particular the polynomial  computed in \cite[page 10]{CGW} and 
	\cite[Section 4]{CGW-arxiv} is what we (and SageMath and 
	\cite{HHL3}) would refer to as 
	$E_{(2,2,1,1,0,0)}(x_6, x_5,\dots,x_1; q,t)$, rather than 
	$E_{(0,0,1,1,2,2)}(x_1, x_2 \dots,x_6; q,t)$.  We have stated 
	\cref{CGWTheorem} so as to be consistent with our conventions
	(and those of SageMath and \cite{HHL3}), so it looks slightly 
	different than the version given in \cite{CGW}.}
\end{thm}

\section{The proofs of \cref{main_result} and \cref{main_result2}}

In this section we prove our main results.  We start by sketching 
an outline of the proofs.

\begin{enumerate}
	\item \label{1} We show that the matrices from 
\cref{defn:matrices} satisfy certain relations generalizing those of the Matrix Ansatz
\cref{eq:MA}, see
		\cref{lem:relations}.

	\item \label{2} We use the relations from \cref{lem:relations} to prove
		that traces of matrix products $\Mat(\mu)$ in $A_0$, $A_1$, $A_2$ satisfy
		a certain recurrence, see 
\cref{thm:ansatz_rec}.  This recurrence allows us to reduce the computation of
traces of matrix products in $A_0$, $A_1$, $A_2$, to the computation of 
traces of matrix products in $A_1$ and $A_2$.

\item \label{3} We show that the weight generating functions for tableaux
	$\Tab_{qtx}(\mu)$ satisfy an analogous recurrence, see 
\cref{thm:tableau}.

\item \label{4} We verify that the base cases (i.e. corresponding to words in 
	$1$'s and $2$'s) agree up to the scalar factor $(1-qt^r)$ where $\mu\in \States(k,r,\ell)$, see 
  \cref{lem:AEword} and 
\cref{lem:AE2}.  It follows that
		$\Tab_{qtx}(\mu) = (1-qt^r) \tr(\Mat(\mu))$.

	\item \label{5} 
Since \cref{lem:relations}
		generalizes the relations of \cref{eq:MA}, 
		\cref{4} and \cref{ansatz} imply that \cref{main_result} holds.

	\item \label{6} Using \cref{4}, it 
		follows that $\Tab_{qtx}(\mu)$ agrees with the quantity
		$f_\mu(x_1,\dots,x_n)$ from \cref{eq:f}, up to normalization.
		
	\item \label{7} To verify that \cref{main_result2} is true
		(i.e. we are getting the actual Macdonald polynomials 
		$E_{\lambda}$ and $P_{\lambda}$ as opposed to scalar 
		multiples of them), we 
		can check the coefficient of $x_{\lambda}$ in $\Tab_{qtx}(\lambda)$
		when $\lambda = 2^{\ell} 1^r 0^k$.
		 There is a unique CRT of type $2^\ell 1^r 0^k$ with 
		 $x$-weight equal to  $x_{\lambda}$; this is the CRT with no 
		 arrows, so its weight is just $x_{\lambda}$.  Similarly, 
		one can check that the coefficient of $x_{\lambda}$ 
		in $E_{\lambda}$ is also 1, for instance by using the formula of Haglund-Haiman-Loehr  \cite{HHL3} and verifying that there is a unique non-attacking
		filling with $x$-weight $x_{\lambda}$.

\end{enumerate}

\subsection{Relations among the matrices from 
\cref{defn:matrices}}

The following lemma gives some relations among the matrices.  Note that \eqref{first} and \eqref{fourth} below
are special cases of \cite[(25) and (27)]{CGW}.  Meanwhile \eqref{second} and \eqref{third} appear somewhat related
to \cite[(26)]{CGW} but are not equivalent to it.

  \begin{lemma}\label{lem:relations}
	  \begin{eqnarray}
		  A_0(x)A_0(y)&=&A_0(y)A_0(x) \label{first} \\
		  A_0(x)A_2(y)&=&tA_2(y)A_0(x)+(1-t)A_2(y)+xy(1-t)A_0(y) \label{second}\\
		  A_0(x)A_1(y)&=&tA_1(y)A_0(x)+(1-t)A_1(y)\label{third}\\
		  A_0(x)S&=&SA_0(qx) \label{fourth}
	  \end{eqnarray}
  \end{lemma}
  \begin{proof}
	  The proof is a series of simple calculations. It suffices to prove 
	  \eqref{second} and \eqref{third}.
	  To prove \eqref{second}, note that 
	  \[
	  (A_0(x)A_2(y))_{i,j}=\left\{ \begin{array}{ll}
		  y^2+xy(1-t^{i+1}) & {\rm if}\ i=j\\
		  xy^2 & {\rm if}\ i=j-1\\
		  y(1-t^i) & {\rm if}\ i=j+1\\
		  0 & {\rm otherwise}
	  \end{array}\right.
	  \] 
	  and also 
	  \begin{multline*}
	  (tA_2(y)A_0(x)+(1-t)A_2(y)+xy(1-t)A_0(y))_{i,j}=\\
	  \left\{ \begin{array}{ll}
		  t(y^2+xy(1-t^{i}))+y^2(1-t)+xy(1-t) & {\rm if}\ i=j\\
		  txy^2+xy(1-t)y & {\rm if}\ i=j-1\\
		  yt(1-t^i)+(1-t)y(1-t^i) & {\rm if}\ i=j+1\\
		  0 & {\rm otherwise}
	  \end{array}\right.
	  \end{multline*}
	  To prove \eqref{third}, note that
	  \[
	  (A_0(x)A_1(y))_{i,j}=\left\{ \begin{array}{ll}
		  yt^i & {\rm if}\ i=j\\
		  xt^{i+1} & {\rm if}\ i=j-1\\
		  0 & {\rm otherwise}
	  \end{array}\right.
	  \]
	  and also 
	  \[
	  (tA_1(y)A_0(x)+(1-t)A_1(y))_{i,j}=\left\{ \begin{array}{ll}
		  yt^{i+1}+(1-t)yt^i & {\rm if}\ i=j\\
		  xt^{i+1} & {\rm if}\ i=j-1\\
		  0 & {\rm otherwise}
	  \end{array}\right.
	  \]
  \end{proof}

\subsection{The recurrence for matrix products}

In this section we give a recurrence for traces of certain matrix products.
We start by verifying a base case.

  \begin{lemma}\label{lem:AEword}
	  Let $\mu\in \{1,2\}^n$ be a composition with $n$ parts which has $2$'s precisely in positions 
$(h_1,\ldots h_\ell)$, and let $W_1 \dots W_n$ be the corresponding matrix product,
with $A_2$'s in positions $h_1,\dots, h_\ell$ and $A_1$'s elsewhere.  Let $r=n-\ell$.
	  Then
	  \[
	  {\tr}(W_1(x_1)\ldots W_n(x_n)S)=\frac{x_1\ldots x_n\prod_{j=1}^\ell x_{h_j}}{1-qt^r}. 
	  \]
  \end{lemma}
  \begin{proof}
	  One can easily check that for $n\ge 1$
	  \[
	  (W_1(x_1)\ldots W_n(x_n)S)_{i,i}=\left\{\begin{array}{ll}
		  x_1^2(W_2(x_2)\ldots W_n(x_n)S)_{i,i}& {\rm if} \ W_1=A_2\\
		  x_1t^i(W_2(x_2)\ldots W_n(x_n)S)_{i,i}& {\rm if} \ W_1=A_1\\
	  \end{array}\right.
	  \]
	  and $S_{i,i}=q^i$.
	  Therefore
	  \begin{eqnarray*}
		  {\tr}(W_1(x_1)\ldots W_n(x_n)S)&=&\sum_i (W_1(x_1)\ldots W_n(x_n))_{i,i}\\
		  &=&\sum_i q^it^{ri}x_1\ldots x_n\prod_{j=1}^\ell x_{h_j}
	  \end{eqnarray*}
  \end{proof}

To state the recurrence, we need some notation.  

\begin{defn}
Given $\mu\in \{0,1,2\}^*$ a word of length $n$ 
in $\States(k,r,\ell)$, we let $\hMat(\mu)$ denote the product of 
matrices obtained from $\mu$ by substituting a $A_0(x_i)$ (respectively
$A_1(x_i)$, $A_2(x_i)$) for 
each $0$, $1$, or $2$ in the $i$th position of $\mu$, and followed by S.
For example, if 
$\mu = 012201$, then 
$\hMat(\mu) = A_0(x_1) A_1(x_2) A_2(x_3) A_2(x_4) A_0(x_5) A_1(x_6)S$.

For $J\subseteq[n]$ and $\mu$ a 
word of length $n$, we let $\hMat(\mu)|_J$ be the subword of 
$\hMat(\mu)$ obtained by restricting to positions $J$.  
For example, 
if $\mu = 012201$, then 
$\hMat(\mu)|_{\{2,4,5\}} = 
 A_1(x_2)  A_2(x_4) A_0(x_5)S$.
\end{defn}

\begin{defn}
Given  $\mu\in \{0,1,2\}^*$,
let $\EE(\mu)=\{i \ \vert \ \mu_i=2\}$ and let 
 $\EA(\mu)=\{i \ \vert \ \mu_i=1 \text{ or }2\}$. 
Given a partial permutation $\sigma \in \Sym_{I, \EA(\mu)}$,
and a choice of $d\in [n]$ such that $\mu_d=0$, 
we define $\tsigma$ to be the sequence obtained from $\sigma$ by 
inserting a $0$ into $\sigma$ in the position that represents the relative position of $\mu_d$ in 
$\EA(\mu)$.  For example, set $\mu = 0121021$ and $d=5$.  Then $\EA(\mu) = \{2,3,4,6,7\}$.
If we choose $I = \{3,6\} \subset \EE(\mu)$, and $\sigma = *\ 2\ *\ 1\ *\ \in \Sym_{I,\EA(\mu)}$, then
$\tsigma = *\ 2\ *\ 0\ 1\ *$. If $I=\emptyset$, we define $\tsigma^{-1}(0)=d$.
\end{defn}

Given $I = \{i_1, i_2,\dots, i_m\} \subset [n]$, we let 
$x_I$ denote $x_{i_1} \dots x_{i_m}$.

\begin{thm}\label{thm:ansatz_rec}
Consider the matrices $A_0(x), A_1(x), A_2(x)$ from \cref{CGW}, 
and let $\mu\in\States(k,r,\ell)$ with $n=k+r+\ell$,
where $r \geq 1$. Suppose $t<1$.
Let $d\in[n]$ be such that $\mu_d=0$.  Then 
we have that $\tr(\Mat(\mu))$ is equal to 
\begin{equation}\label{mat:recurrence}
\sum_{I\subseteq \EE(\mu)} 
	\frac{ [r+\ell-|I|]_{qt}!}{[r+\ell]_{qt}!} (x_I)^2 x_d \tr(\Mat(\mu)|_{[n]\backslash I \cup \{d\}}) 
	\sum_{\sigma\in \Sym_{I,\EA(\mu)}} \frac{t^{\dis(\tsigma)}q^{\rec(\tsigma)}}{x_{\tsigma^{-1}(|I|)}}.
\end{equation}

 \end{thm}
 
 Note that by \cref{def:tsigma} we have $x_{\tsigma^{-1}(0)}=x_d$, and so $I=\emptyset$ gives the term $\tr(\Mat(\mu)|_{[n]\backslash \{d\}})$ in the above.

\begin{example} If $\mu=0212$ (so that $k=1$, $r=1$, $\ell=2$), and $d=1$, then 
\cref{thm:ansatz_rec} says that 
\begin{align*}
\hspace{4em}&\hspace{-4em}\tr(A_0(x_1) A_2(x_2) A_1(x_3) A_2(x_4)S)= 
\tr(A_2(x_2) A_1(x_3) A_2(x_4)S)\\
& + 
[r+\ell]_{qt}( \tr( A_1(x_3) A_2(x_4)S) x_1x_2 t^{\dis(1**)} 
 + \tr(A_2(x_2) A_1(x_3)S) x_1x_4 t^{\dis(**1)})\\
&+
[r+\ell]_{qt}[r+\ell-1]_{qt} \tr(A_1(x_3)S) (x_1x_2 t^{\dis(1*2)} + x_1x_4 q t^{\dis(2*1)}).
\end{align*}
\end{example}

\begin{proof}
In the expression $\Mat(\mu)$, we replace the $A_0$ in position $d$ by a $\tilde{A_0}$ so that 
we can keep track of this 
``marked" $A_0$.  
Without loss of generality, using the fact that 
$\tr(M_1 \dots M_n) = \tr(M_2 \dots M_n M_1)$, we can assume 
that $d=1$.

Using \eqref{eq:MA}, we will 
apply the operations below (and only these ones)
to $\tr(\Mat(\mu))$ until $\tilde{A_0}$ is annihilated in every term on the right-hand side.
\begin{align}
\tilde{A_0}(x)A_2(y) &= tA_2(y)\tilde{A_0}(x)+xy(1-t)\tilde{A_0}(y)+(1-t)A_2(y) \label{a:1} \\
\tilde{A_0}(x)A_1(y) &= tA_1(y)\tilde{A_0}(x)+(1-t)A_1(y) \label{a:2}\\
\tilde{A_0}(x)A_0(y) &= A_0(y)\tilde{A_0}(x) \label{a:3}\\
\tilde{A_0}(x)S &= S \tilde{A_0}(qx). \label{a:4}
\end{align}

More specifically, we think of \eqref{a:1} as giving us the choice of either 
moving the $\tilde{A_0}$ to the right past an $A_2$ (picking up a factor of 
$t$), or annihilating an $A_2$ or annihilating 
the $\tilde{A_0}$ (in each case picking up a factor of $(1-t)$). 
Similarly, \eqref{a:2} gives us the choice of moving the $\tilde{A_0}$
to the right past an $A_1$ picking up a factor of $t$, or annihilating the $\tilde{A_0}$ and picking up a factor of $(1-t)$.  \eqref{a:3} allows us to 
move the $\tilde{A_0}$ to the right past a $A_0$, and, if we have moved the $\tilde{A_0}$
to the end of the word, \eqref{a:4} allows us to move it back to the beginning.

After applying \eqref{a:1} through \eqref{a:4} as long as possible, we will  be
left with terms obtained from $\tr(\Mat(\mu))$ by:
\begin{itemize}
\item deleting some subset $I\subseteq\EE(\mu)$ of the $A_2$'s, having 
chosen a certain order $\sigma \in \Sym_{I, \EA(\mu)}$ in which to delete them
\item either deleting or not deleting the $\tilde{A_0}$; in the latter case, that means 
 that we wind up commuting the $\tilde{A_0}$ past all the remaining $A_1$ and $A_2$ letters of $\mu|_{[n]\backslash I}$ 
 infinitely many times.
\end{itemize}

We obtain that $\tr(\Mat(\mu))$ is equal to the following:
\begin{multline} \label{eq_red} 
 \sum_{I\subseteq \EE(\mu)} (x_I)^2 x_d\\ 
\cdot \Bigg[ \sum_{\sigma\in \Sym_{I,\EA(\mu)}} \frac{1}{x_{\tsigma^{-1}(|I|)}}
 \prod_{i=0}^{|I|-1}(1-t)(1+q t^{r+\ell-i}+q^2 t^{2(r+\ell-i)}+\cdots)
t^{\dist_{i+1}(\tsigma)} q^{\cyc_{i+1}(\tsigma)}   \Bigg]
\end{multline}
\begin{multline} \label{eq_blue}
\cdot \Bigg[  \sum_{m=1}^{r+\ell-|I|} 
(1-t)(1+t^{r+\ell-|I|}+t^{2(r+\ell-|I|)}+\cdots)t^{m-1}  \tr(\Mat(\mu)|_{[n]\backslash I \cup \{d\}})
\\
+  \lim_{j \to \infty} t^{j(r+\ell-|I|)} \tr(\Mat(\mu)|_{[n]\backslash I})(x_d \to q^jx_d)  \Bigg]. 
\end{multline}
We use the notation $\delta_{(*)}$ to represent the Kronecker delta, which returns 1 if $(*)$ is true and 0 otherwise, and $\cyc_{i+1}(\tsigma)=\delta_{(\tsigma^{-1}(i+1)<\tsigma^{-1}(i))}$.

Let us explain  the factor in \eqref{eq_red}: 
this is the factor we pick up in deleting the chosen (possibly empty)
set $I\subseteq\EE(\mu)$ of $A_2$'s. If $I=\emptyset$, we simply get 1. Otherwise,
suppose we delete $|I|=s>0$ $A_2$'s, in the order and positions specified
by $\sigma\in\Sym_{I,\EA(\mu)}$. Let $u_i=\tsigma^{-1}(i)$ be the label of the $i$'th $A_2$ to be deleted.
We start by deleting the $A_2$ with label $u_1$: to do so,
we first  commute the $\tilde{A_0}$ 
past all letters of the word $\mu$ a total of $j_1$ times (where $j_1 \geq 0$),
thus picking up a factor of $t^{j_1(r+\ell)}$ with $\tilde{A_0}(x)$ becoming $\tilde{A_0}(q^{j_1}x)$; we then apply some number
$m < r+\ell$ of commutations to bring the $\tilde{A_0}$ adjacent to 
this $A_2$. Note that $m=\dist_1(\tsigma)$, and 
 so we pick up a factor of $t^{\dist_1(\tsigma)}$ with $\tilde{A_0}(x)$ becoming $\tilde{A_0}(q^{\cyc_1(\tsigma)}x)$.
  We then delete the $A_2$ with label $u_{1}$, picking up a factor of $x_{u_{1}} x_{d} q^{j_1}q^{\cyc_1(\tsigma)}(1-t)$.
Similarly, to delete the $A_2$ with label $u_2$, we commute the $\tilde{A_0}$
past all remaining letters of the word $\mu$ a total of $j_2\geq 0$ times,
picking up $t^{j_2(r+\ell-1)}$ and with $\tilde{A_0}(x)$ becoming $\tilde{A_0}(q^{j_2}x)$, then apply $\dist_2(\tsigma)$ commutations to 
move the $\tilde{A_0}$ from position $u_1$ to $u_2$ with $\tilde{A_0}(x)$ becoming $\tilde{A_0}(q^{\cyc_2(\tsigma)}x)$.
 We then delete
that $A_2$, picking up a factor of $x_{u_{2}} x_{u_1} q^{j_2}q^{\cyc_2(\tsigma)}(1-t)$.  We continue in this fashion until the last $A_2$, which has label $u_s$: when this is deleted, we pick up a factor of $x_{u_s} x_{u_{s-1}} q^{j_{s}}q^{\cyc_{s}(\tsigma)}(1-t)$. Thus the overall contribution of the $x$'s is $\frac{(x_I)^2x_d}{x_{\tsigma^{-1}(|I|)}}$, which is how we obtain the factor in \eqref{eq_red}.  

After annihilating the chosen $A_2$'s, we then either delete
the $\tilde{A_0}$, or we don't. If we do delete the $\tilde{A_0}$, we obtain the sum in the first line of \eqref{eq_blue}.  Again we possibly cycle the $\tilde{A_0}$ through all
the remaining  $r+\ell-|I|$ $A_1$ and $A_2$ letters 
of $\mu$ $j$ times, then commute it past $m$
more letters, where $0 \leq m \leq r+\ell-|I|-1$. Note that here, even though the $\tilde{A_0}(x)$ does become $\tilde{A_0}(q^jx)$, there are no further components arising from the term $xy(1-t)\tilde{A_0}(y)$ in \eqref{a:2}; thus the variable that the $\tilde{A_0}$ carries never enters into the equation, and so no $q$'s are collected in the final expression.

If we don't ultimately delete the $\tilde{A_0}$, then 
we necessarily cycle the 
$\tilde{A_0}$ around the remaining letters of $\mu$ indefinitely, resulting
in the term 
$ \lim_{j \to \infty} t^{j(r+\ell-s)} \tr(\Mat(\mu)|_{[n]\backslash I})(x_{d} \to q^jx_d),$ 
where the notation $x_d \to q^j x_d$ means that we substitute $q^j x_d$ for $x_d$ in $\tilde{A_0}$; this is the second line of \eqref{eq_blue}.

Now we can simplify the terms within the sums obtained above.
Since $t<1$, the terms involving the limit go to $0$.
In the top line of \eqref{eq_blue}, we have that 
 $\sum_{m=1}^{r+\ell-s} (1-t)(1+t^{r+\ell-s}+\dots)t^{m-1}$ is equal to $1$.
To simplify the bottom line of \eqref{eq_red},
we recall that  $\sum_{i=0}^{|I|-1} \dist_{i+1}(\tsigma) = 
\dis(\tsigma)$ and that by definition $\sum_{i=0}^{|I|-1} \cyc_{i+1}(\tsigma) = \rec(\tsigma)$.

We thus obtain the desired identity \eqref{mat:recurrence}.
\end{proof}

Observe that at $q=x_1=\cdots=x_n=1$, the recurrence \eqref{mat:recurrence} becomes
\begin{equation}
\tr(\Mat(\mu)) =  \sum_{I\subseteq \EE(\mu)} \frac{[r+\ell-|I|]!}{[r+\ell]!} \tr(\Mat(\mu)|_{[n]\backslash I\cup\{d\}}) \sum_{\sigma\in \Sym_{I,\EA(\mu)}} t^{\dis(\tsigma)}.
\end{equation}

\begin{remark} The case $k=0$ is trivial, since in this case 
	the stationary distribution of the  ASEP is uniform.  From 
	\cref{lem:AEword}  we obtain $\tr(\Mat(\mu))=\frac{1}{1-qt^n}x_{\EA(\mu)}x_{\EE(\mu)}$.
We also get the uniform distribution at $t=1$.
\end{remark}

\begin{remark}\label{rem:sufficient}
We won't actually need the full generality of \cref{thm:ansatz_rec}; it is enough to 
know \cref{thm:ansatz_rec} in the case that $d\in [n]$ is maximal such that $\mu_d=0$.
\end{remark}

\subsection{The recurrence for weight generating functions of tableaux}

We again start by giving a base case.
\begin{lemma}\label{lem:AE2}
Let $\mu\in \States(0,r,\ell)$ with $r+\ell=n$, 
i.e. it is a composition with parts equal to $1$ or $2$.  Let $h_1,\dots,h_\ell$ be the 
positions of the $2$'s.  Then we have 
\[
\Tab_{qtx}(\mu) = x_1\dots x_n \prod_{j=1}^{\ell} x_{h_j}.
\] 
\end{lemma}
\begin{proof}
\cref{lem:AE2} follows directly from 
the definitions and in particular \cref{weight_def}.
\end{proof}

In what follows, if $\mu\in \States(k,r,\ell)$ with $k+r+\ell=n$, and if $J\subset [n]$,
then we let $\Tab_{qtx}(\mu|_J)$ be the weight generating function for cylindric rhombic tableaux
of type $\mu|_J$ obtained if we label the path $P(\mu)$ in each tableau using the numbers $J$.
(So that the $x$-weight of each tableau is a monomial in $x_{j}$'s for $j\in J$.)

\begin{thm}\label{thm:tableau}
Let $\mu\in\States(k,r,\ell)$ with $n=k+r+\ell$. Let $d\in[n]$ be maximal such that $\mu_d=0$. 
Then 
$\Tab_{qtx}(\mu)$ equals
\begin{equation*}
 x_d \sum_{s=0}^{\ell} \frac{[r+\ell-s]!}{[r+\ell]!}\sum_{\substack{I\subseteq \EE(\mu)\\|I|=s}} x_I^2 \Tab_{qtx}(\mu|_{[n]\backslash I \cup \{d\}}) \sum_{\sigma^1\in \Sym_{I,\EA(\mu)}}\frac{t^{\dis(\tilde{\sigma}^1)}q^{\rec(\tsigma^1)}}{x_{(\tsigma^1)^{-1}(s)}}.
\end{equation*}
\end{thm}

\begin{proof}
We prove that  \cref{weight_def} satisfies the recurrence using
a bijective proof.
Since $d\in[n]$ is maximal such that $\mu_d=0$, removing $\mu_d$ from $\mu_1 \dots \mu_n$
corresponds to deleting the bottom row $\west(1)$ from any $T\in \CRT(\mu)$. 

Choose some $T\in \CRT(\mu)$ with $k$ rows. 
 Suppose that $\west(1)$ contains $s_1$ up-arrows in columns with positions
corresponding to  $I\subseteq \EE(\mu)$. Define $\hat{T}$ to be the tableau with $k-1$ rows obtained by removing 
 $\west(1)$ as well as the $s_1$ columns corresponding to $I$, 
and then gluing together the remaining boxes in the obvious way. 

If $s_1=0$, $\hat{T}$ is simply the same tableau with row labeled $d$ removed, whose weight is $\frac{x_d}{x_{(\tsigma^1)^{-1}(0)}}\Tab_{qtx}(\mu|_{[n]\backslash\{d\}})$, which is simply $\Tab_{qtx}(\mu|_{[n]\backslash\{d\}})$. 

When $s_1\geq 1$, clearly
$\hat{T} \in \CRT(\mu|_{[n]\backslash I\cup\{d\}})$, and in fact the set of $T\in \CRT(\mu)$ with 
$s_1$ up-arrows in locations $I$ in $\west(1)$ maps bijectively to the set $\CRT(\mu|_{[n]\backslash I\cup\{d\}})$. 
Moreover if we choose an arrow ordering $\{\sigma^i\} = \{\sigma^1,\ldots,\sigma^k\}$ for $T$,
then this induces an arrow ordering $\{\hat{\sigma}^i\}=\{\sigma^2,\ldots,\sigma^k\}$ for $\hat{T}$, with the property that $\dis(\hat{\sigma}^i)=\dis(\sigma^{i+1})$ and $\cyc(\hat{\sigma}^i)=\cyc(\sigma^{i+1})$ for $i=1,\ldots,k-1$.
 Therefore 
$\sum_T \wt_{qtx}(T)$, where the sum is over $T\in\CRT(\mu)$ 
with $s_1$ arrows in locations $I$ of $\west(1)$, is equal to 
$x_d \Tab_{qtx}(\mu|_{[n]\backslash I\cup\{d\}})$ times the contribution of weights from all 
possible orderings $\sigma^1\in\Sym_{I,\EA(\mu)}$. 
By \cref{weight_def}, the possible choices of  arrow orderings 
contribute 
\[
\sum_{\sigma^1\in \Sym_{I,\EA(\mu)}} t^{\dis(\tsigma^1)}q^{\rec(\tsigma^1)} \frac{x_I^2}{x_{\tsigma^{-1}(s_1)}}
\] 
to the weight.
\end{proof}

\subsection{Symmetries of the tableaux}

The following statements can be proved using \cref{main_result} and the symmetries of the 
Markov chain $\2ASEP(k,r,\ell)$.  However, it is not obvious how to give a combinatorial proof 
using the tableaux.

\begin{problem}\label{conj:complement}
For any $\mu\in\States(k,r,\ell)$, 
define $\mu^{\odot}\in\States(\ell,r,k)$ to be the \emph{complement} of $\mu$, which is the word obtained by replacing each $0$ by 
a $2$ and vice-versa.
 For example, for $\mu=2210$, $\mu^{\odot}=0012$. Give a combinatorial proof that
\[
 \Tab_t(\mu)=t^{\ell(\ell+2r-1)/2}\Tab_{1/t}(\mu^{\odot}).
\]
(Note that $\frac{(\ell)(\ell+2r-1)}{2}$ is the degree of $\Tab_t(\mu)$.)
\end{problem}

\begin{problem}\label{conj:conjugate}
For any $\mu=\mu_1\dots \mu_n \in\States(k,r,\ell)$, 
 define $\mu^{T}:=\mu_n^{\odot}\ldots \mu_1^{\odot}$ to be the ``particle-hole symmetry" word.
For example, for $\mu=2210$, we have $\mu^{T}=2100$. Give a combinatorial proof that
\[
\Tab_t(\mu)[r+k]_t!=\Tab_t(\mu^T)[r+\ell]_t!
\]
\end{problem}

We next compute $\Tab_t(\mu)$ when $t=1$. 

\begin{cor}
Let $\mu\in\States(k,r,\ell)$ with $k+r+\ell=n$. Then
$\Tab_t(\mu)(t=1)={n \choose \ell}\frac{\ell!r!}{(r+\ell)!}$.
\end{cor}

\begin{proof}
For $\mu\in\States(k,r,\ell)$, the 
\emph{$\mu$-diagram} $\mathcal{H}(\mu)$ has 
$k$ rows.   
Each row in $\mathcal{H}(\mu)$ contains  $\ell$ square tiles, each of which is either 
empty or contains an arrow.
Since we are setting $t=1$, we do not need to compute 
the disorder of any arrow placements, we simply need to determine
how many arrow placements and arrow orderings there are.

Suppose we are selecting an arrow placement for $\mathcal{H}(\mu)$.
We first choose the total number $s$ of arrows to place in the square tiles, where 
 $0\leq s\leq \ell$;
 there are ${\ell \choose s}$ choices for the $s$ columns that will contain these arrows. 
Let $s_1+\cdots+s_k=s$ be a composition representing the number of arrows placed in 
the rows $\west(1), \dots, \west(k)$. 
Once $s_1,\ldots,s_k$ are chosen (in ${s+k-1 \choose k-1}$ ways), 
there are ${s\choose s_1,\ldots,s_k}$ ways to select which arrows go in which rows, and $s_i!$ possible orderings of the arrows in $\west(i)$, for each $i\in\{1,\ldots,k\}$. 
Finally, given that $\arr(T)=s$, we have that the factor
	$\frac{[r+\ell-\arr(T)]!}{[r+\ell]!}$
in \cref{tweight_def} is equal to 
$\frac{(r+\ell-s)!}{(r+\ell)!}$. Thus we obtain
\begin{align*}
\Tab_t(\mu)(t=1)& = \sum_{0\leq s\leq\ell} \sum_{s_1+\cdots+s_k=s} {\ell \choose s} {s\choose s_1,\ldots,s_k}s_1!\ldots s_k!  \frac{(r+\ell-s)!}{(r+\ell)!} \\
& = \sum_{0\leq s\leq\ell} {s+k-1 \choose k-1}{\ell \choose s} \frac{s!(r+\ell-s)!}{(r+\ell)!}\\
&= \frac{r!\ell!}{(r+\ell)!}\sum_{0\leq s\leq \ell} {s+k-1 \choose k-1}{r+\ell-s \choose r} \\
&=\frac{r!\ell!}{(r+\ell)!}{k+r+\ell \choose \ell} = \frac{r!\ell!}{(r+\ell)!}{n \choose \ell}.
\end{align*}
\end{proof}


\section{A bijection from cylindric rhombic tableaux to two-line queues}

In this section we present a bijection between cylindric rhombic
tableaux and two-line queues which are equivalent to the multiline queues of  Martin \cite{Martin}.  

\begin{defn}
A two-line queue of size $n$ is a two rowed array $Q$ on a cylinder where the entries can be $\{ \tikzcircle{5pt},\tikzcircle[fill=black]{1pt}\}$ where $\tikzcircle{5pt}$ means that there is a ball
at the site and $\tikzcircle[fill=black]{1pt}$ means the site is empty.
There exists
a partial matching between the balls in the top row and the bottom row
such that:
\begin{itemize}
\item All the balls of the top row are matched
\item A ball in the bottom row is allowed to not be matched only if there is no ball in the same column in the top row.
\end{itemize}
\end{defn}
For each matching of a top row ball from column $i$ to a bottom row ball in column $j$, we draw an edge
from left to right, wrapping around if necessary. See \cref{exampl}. In this queue, the top row balls in columns 2, 3, 6, 8, 10 are matched with the bottom row balls in columns 2, 10, 9, 4, 5, respectively. 

The \emph{type} of a two-line queue is a word in $\{0,1,2\}^*$ which is read off the bottom row from left to right: an empty site is read as a 0, an unmatched ball is read as a 1, and a matched ball is read as a 2. The type of the queue in \cref{exampl} is 0212201022.

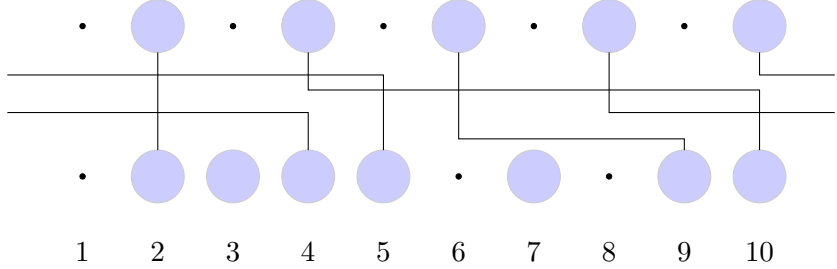
\begin{figure}
\begin{tikzpicture}
\draw[black,fill=black,radius=1pt](0,0) circle;
\draw[black,fill=blue, opacity=.2,radius=10pt](1,0) circle;
\draw[black,fill=blue, opacity=.2,radius=10pt](3,0) circle;
\draw[black,fill=black,radius=1pt](2,0) circle;
\draw[black,fill=black,radius=1pt](4,0) circle;
\draw[black,fill=blue, opacity=.2,radius=10pt](5,0) circle;\draw[black,fill=black,radius=1pt](6,0) circle;
\draw[black,fill=blue, opacity=.2,radius=10pt](7,0) circle;\draw[black,fill=black,radius=1pt](8,0) circle;
\draw[black,fill=blue, opacity=.2,radius=10pt](9,0) circle;

\draw[black,fill=black,radius=1pt](0,-2) circle;
\draw[black,fill=blue, opacity=.2,radius=10pt](1,-2) circle;

\draw[black,fill=blue, opacity=.2,radius=10pt](2,-2) circle;
\draw[black,fill=blue, opacity=.2,radius=10pt](3,-2) circle;
\draw[black,fill=blue, opacity=.2,radius=10pt](4,-2) circle;
\draw[black,fill=black,radius=1pt](5,-2) circle;
\draw[black,fill=blue, opacity=.2,radius=10pt](6,-2) circle;
\draw[black,fill=black,radius=1pt](7,-2) circle;
\draw[black,fill=blue, opacity=.2,radius=10pt](8,-2) circle;
\draw[black,fill=blue, opacity=.2,radius=10pt](9,-2) circle;

\draw (1,-.35)--(1,-1.65);
\draw (3,-.35)--(3,-.85)--(9,-.85)--(9,-1.65);
\draw (5,-.35)--(5,-1.5)--(8,-1.5)--(8,-1.65);
\draw (9,-.35)--(9,-.65)--(10,-.65);
\draw(-1,-.65)--(4,-.65)--(4,-1.65);
\draw (7,-.35)--(7,-1.15)--(10,-1.15);
\draw(-1,-1.15)--(3,-1.15)--(3,-1.65);

\node at (0,-3) {1};
\node at (1,-3) {2};
\node at (2,-3) {3};
\node at (3,-3) {4};
\node at (4,-3) {5};
\node at (5,-3) {6};
\node at (6,-3) {7};
\node at (7,-3) {8};
\node at (8,-3) {9};
\node at (9,-3) {10};

\end{tikzpicture}
\caption{An example of a two-line queue of type 0212201022.}
\label{exampl}
\end{figure}

To each queue, we associate a weight in $x_1,\ldots ,x_n,q,t$. Each ball in column $i$ has weight $x_i$. We also give a weight to the edges that connect balls in different columns.
We explore the queue with a simple algorithm. We call a ball \emph{restricted} if it has another ball besides itself in its column. 
\begin{enumerate}
\item At initialization, all bottom row balls are considered \emph{free}, and all balls are unmatched.
\item Let $i$ be the column containing the rightmost unrestricted top row ball that has not yet been matched. If there are no remaining unmatched unrestricted top row balls, we are done.
\item To compute the weight of a matching from the top row ball in column $i$, let {\rm free} be the number of free bottom row balls remaining at this point. Suppose the ball in column $i$ is matched to the bottom row ball in column $j$. Then {\rm skipped} is the number of free bottom row balls that are skipped over to get from column $i$ to column $j$ while moving to the right, wrapping around if necessary. The weight of that matching is 
\[
\frac{q^{\delta_{(i>j)}}t^{\rm skipped}}{[{\rm free}]_{qt}},
\]
where $\delta$ denotes the Kronecker delta. In other words,
\begin{itemize}
\item if $i< j$, the weight of that matching is $t^{\rm skipped}/[{\rm free}]_{qt}$ where {\rm skipped} is the number of free bottom row balls in columns $u$ such that $i \leq u<j$.
\item if $i>j$, the weight of that matching is 
$qt^{\rm skipped}/[{\rm free}]_{qt}$ where {\rm skipped} is the number of free bottom row balls in columns $u$ such that $u \geq i$ or $u<j$.
\end{itemize}
The bottom row ball in column $j$ that has been matched is now no longer free.
\item If there is a top row ball in column $j$, continue to Step 3, setting $i=j$. Otherwise, go to Step 2.
\end{enumerate}
Now the {\em weight of the two-line queue}
is the product of the weight of the edges times the weight of the balls.

For example, let us compute the weight of the queue in \cref{exampl}. We start with the top row ball in column 8. It is matched with the bottom row ball in column 4, by cycling around and skipping 4 free balls out of a total of 7 free balls. The weight of this edge is thus $qt^4/[7]_{qt}$. 

Since there is a top row ball in column 4, that is the next one we match. This ball is matched
with the bottom row ball in column 10, by skipping 3 free balls out of a total of 6 remaining free balls.
The weight of this edge is thus $t^3/[6]_{qt}$. 

We continue with the top row ball in column 10. It is matched
with the bottom row ball in column 5 by cycling around and skipping 2 free balls out of a total of 5 remaining free balls. The weight of this edge is thus $qt^2/[5]_{qt}$.

The next ball to be matched is the rightmost unrestricted unmatched top row ball, which is in column 6: this one is matched to the bottom row ball in column 9 by skipping 1 free ball out of a total of 4 remaining free balls. The weight of this edge is $t/[4]_{qt}$.

There are no remaining unmatched unrestricted top row balls, so therefore the weight of this two-line queue is
\[
\frac{q^2t^{10}x_2^2x_3x_4^2x_5x_6x_7x_8x_9x_{10}^2}{[7]_{qt}[6]_{qt}[5]_{qt}[4]_{qt}}.
\]

\begin{remark}
There is some recent work \cite{AasGrinbergS} that considers the usual multiline queues (at $t=0$) with $\{x_i\}$ weights as we have defined them here; in that paper, the authors call them \emph{multiline queues with spectral parameters}. 
\end{remark}

We will exhibit a construction that will prove:
\begin{thm}
There exists a bijection between CRTs of type $\mu\in\States(k,r,\ell)$
and two-line queues of type $\mu$. This bijection is weight preserving.
\end{thm}

%
%

\begin{figure}[!ht]
  \centerline{\includegraphics[width=2in]{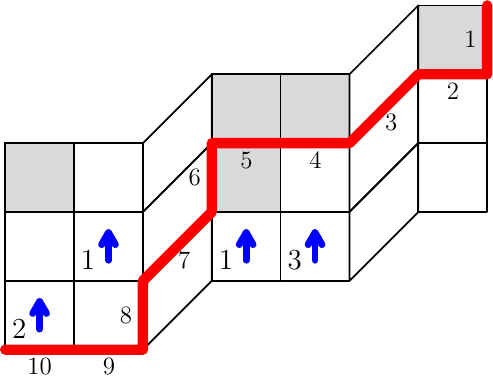}}
\centering
 \caption{A CRT of type $\mu=0212201022$ that corresponds to the two-line queue of \cref{exampl}.}
\label{CRT}
 \end{figure}

\begin{proof}
We present the bijection. Given a CRT $T$ of type $\mu\in\States(k,r,\ell)$, we label the rows and columns of $T$ by the label of the corresponding edge in $P(\mu)$. We build a queue $Q$ from $T$. We first fill the two rows of the queue with the following rules. For each $i$, the site in column $i$ of the bottom row is empty if and only if $\mu_i=0$; otherwise it contains a ball. 
The site in column $i$ of the top row is empty if and only if one of the following occurs: 
\begin{itemize}
\item $\mu_i=1$, 
\item $\mu_i=0$ and the row $i$ of $T$ is empty, or 
\item $\mu_i=2$ and the column $i$ of $T$ contains an arrow which has the largest label in its row.
\end{itemize}

We now explain how to match the balls in $Q$. 
\begin{itemize}
\item A restricted top row ball in column $i$ is matched to the bottom row ball in column $i$ if and only if the column $i$ of $T$ is empty.
\item An unrestricted top row ball in column $i$ is matched to the the bottom row ball in column $j$ if and only if there exists in $T$ an arrow labelled $1$ in row $i$ and column $j$.
\item A restricted top row ball in column $i$ is matched to the bottom row ball in column $j$ where $i\neq j$ if and only if there exists in $T$ an arrow labelled with some $k>1$ in column $j$, and in the same row there is an arrow
labelled $k-1$ in column $i$.
\end{itemize}

It is a simple exercise to check that the weight of the $Q$ is equal to the weight of $T$, and the construction
is bijective.
\end{proof}

{\bf Example of the bijection.} We start with the CRT $T$ in \cref{CRT}, where we have labeled the edges of $P(\mu)$ from 1 to 10 to correspond with the labels of the columns of the two-line queue $Q$. 
We first fill the bottom row of $Q$ by putting balls in all sites except for those in columns 1, 6, and 8, since those correspond to the vertical edge labels in $T$. 

Now we fill the top row of $Q$. We put an empty site in column 1, as row 1 of $T$ is empty.
The edges 3 and 6 or $P(\mu)$ are diagonal, so we put an empty site in columns 3 and 6 of $Q$.
Finally the columns 5 and 9 of $T$ contain an arrow with the largest label in its corresponding row, and therefore we put an empty site in columns 5 and 9 of $Q$. The rest of the sites are filled with balls.

We proceed to match balls between the two rows of $Q$, starting with the empty columns of $T$. Column 2 of $T$ is empty, so the top row ball in column 2 is matched to the bottom row ball in column 2. Now we look at the non-empty rows of $T$ from bottom to top.

In row 8 and column 4 of $T$, there is an arrow labelled 1: therefore the top row ball in column 8 or $Q$ is matched with the bottom row ball in column 4.

In row 8 and column 10 of $T$, there is an arrow labelled 2: therefore the top row ball in column 4 of $Q$ is matched to the bottom row ball in column 10.

In row 8 and column 5 of $T$, there is an arrow labelled 3: therefore the top row ball  in column 10 of $Q$ is matched to the bottom row ball in column 5.

We now look at row 6 of $T$. In column 9, there is an arrow labeled 1: therefore the top row ball in column 6 of $Q$ is matched to the bottom row ball in column 9.

Having recorded all the arrows in $T$, we get the two-line queue of \cref{exampl}.

\bibliographystyle{alpha}
\bibliography{bibliography}

\begin{thebibliography}{AAMP12}

\bibitem[AAMP12]{AritaAyyerMallickProlhac12}
Chikashi Arita, Arvind Ayyer, Kirone Mallick, and Sylvain Prolhac.
\newblock Generalized matrix ansatz in the multispecies exclusion process---the
  partially asymmetric case.
\newblock {\em J. Phys. A}, 45(19):195001, 16, 2012.

\bibitem[AGS18]{AasGrinbergS}
Erik Aas, Darij Grinberg, and Travis Scrimshaw.
\newblock Multiline queues with spectral parameters.
\newblock 2018.
\newblock arXiv:1810.08157.

\bibitem[AL14]{AyyerLinusson14}
Arvind Ayyer and Svante Linusson.
\newblock An inhomogeneous multispecies {TASEP} on a ring.
\newblock {\em Adv. in Appl. Math.}, 57:21--43, 2014.

\bibitem[AL18]{AasLinusson18}
Erik Aas and Svante Linusson.
\newblock Continuous multi-line queues and {TASEP}.
\newblock {\em Ann. Inst. Henri Poincar\'{e} D}, 5(1):127--152, 2018.

\bibitem[AM13]{AritaMallick12}
Chikashi Arita and Kirone Mallick.
\newblock Matrix product solution of an inhomogeneous multi-species {TASEP}.
\newblock {\em J. Phys. A}, 46(8):085002, 11, 2013.

\bibitem[Ang06]{Angel}
Omer Angel.
\newblock The stationary measure of a 2-type totally asymmetric exclusion
  process.
\newblock {\em J. Combin. Theory Ser. A}, 113(4):625--635, 2006.

\bibitem[BC14]{BorodinCorwin}
Alexei Borodin and Ivan Corwin.
\newblock Macdonald processes.
\newblock {\em Probab. Theory Related Fields}, 158(1-2):225--400, 2014.

\bibitem[BE04]{BE}
R.~Brak and J.~W. Essam.
\newblock Asymmetric exclusion model and weighted lattice paths.
\newblock {\em J. Phys. A}, 37(14):4183--4217, 2004.

\bibitem[Can17]{Cantini}
Luigi Cantini.
\newblock Asymmetric simple exclusion process with open boundaries and
  {K}oornwinder polynomials.
\newblock {\em Ann. Henri Poincar\'e}, 18(4):1121--1151, 2017.

\bibitem[CdGW]{CGW-arxiv}
Luigi Cantini, Jan de~Gier, and Michael Wheeler.
\newblock Matrix product and sum rule for {M}acdonald polynomials.
\newblock FPSAC abstract.

\bibitem[CdGW15]{CGW}
Luigi Cantini, Jan de~Gier, and Michael Wheeler.
\newblock Matrix product formula for {M}acdonald polynomials.
\newblock {\em J. Phys. A}, 48(38):384001, 25, 2015.

\bibitem[Che95a]{Cher2}
Ivan Cherednik.
\newblock Double affine {H}ecke algebras and {M}acdonald's conjectures.
\newblock {\em Ann. of Math. (2)}, 141(1):191--216, 1995.

\bibitem[Che95b]{Cher1}
Ivan Cherednik.
\newblock Nonsymmetric {M}acdonald polynomials.
\newblock {\em Internat. Math. Res. Notices}, (10):483--515, 1995.

\bibitem[CMW17]{CMW}
Sylvie Corteel, Olya Mandelshtam, and Lauren Williams.
\newblock Combinatorics of the two-species {ASEP} and {K}oornwinder moments.
\newblock {\em Adv. Math.}, 321:160--204, 2017.

\bibitem[CMW18]{CMW-MLQ}
Sylvie Corteel, Olya Mandelshtam, and Lauren Williams.
\newblock From multiline queues to macdonald polynomials via the exclusion
  process.
\newblock 2018.
\newblock arXiv:?

\bibitem[CW07]{CW1}
Sylvie Corteel and Lauren~K. Williams.
\newblock Tableaux combinatorics for the asymmetric exclusion process.
\newblock {\em Adv. in Appl. Math.}, 39(3):293--310, 2007.

\bibitem[CW11]{CW-Duke1}
Sylvie Corteel and Lauren~K. Williams.
\newblock Tableaux combinatorics for the asymmetric exclusion process and
  {A}skey-{W}ilson polynomials.
\newblock {\em Duke Math. J.}, 159(3):385--415, 2011.

\bibitem[CW15]{CW-Koornwinder}
Sylvie Corteel and Lauren Williams.
\newblock Macdonald-{K}oornwinder moments and the two-species exclusion
  process.
\newblock to appear in Selecta Mathematica, 2015.

\bibitem[DEHP93]{DEHP}
B.~Derrida, M.~R. Evans, V.~Hakim, and V.~Pasquier.
\newblock Exact solution of a {$1$}{D} asymmetric exclusion model using a
  matrix formulation.
\newblock {\em J. Phys. A}, 26(7):1493--1517, 1993.

\bibitem[DJLS93]{Ansatz1}
B.~Derrida, S.~A. Janowsky, J.~L. Lebowitz, and E.~R. Speer.
\newblock Exact solution of the totally asymmetric simple exclusion process:
  shock profiles.
\newblock {\em J. Statist. Phys.}, 73(5-6):813--842, 1993.

\bibitem[DS05]{jumping}
Enrica Duchi and Gilles Schaeffer.
\newblock A combinatorial approach to jumping particles.
\newblock {\em J. Combin. Theory Ser. A}, 110(1):1--29, 2005.

\bibitem[EFM09]{EvansFerrariMallick08}
Martin~R. Evans, Pablo~A. Ferrari, and Kirone Mallick.
\newblock Matrix representation of the stationary measure for the multispecies
  {TASEP}.
\newblock {\em J. Stat. Phys.}, 135(2):217--239, 2009.

\bibitem[FM07]{FerrariMartin}
Pablo~A. Ferrari and James~B. Martin.
\newblock Stationary distributions of multi-type totally asymmetric exclusion
  processes.
\newblock {\em Ann. Probab.}, 35(3):807--832, 2007.

\bibitem[HHL05a]{HHL1}
J.~Haglund, M.~Haiman, and N.~Loehr.
\newblock A combinatorial formula for {M}acdonald polynomials.
\newblock {\em J. Amer. Math. Soc.}, 18(3):735--761, 2005.

\bibitem[HHL05b]{HHL2}
J.~Haglund, M.~Haiman, and N.~Loehr.
\newblock Combinatorial theory of {M}acdonald polynomials. {I}. {P}roof of
  {H}aglund's formula.
\newblock {\em Proc. Natl. Acad. Sci. USA}, 102(8):2690--2696, 2005.

\bibitem[HHL08]{HHL3}
J.~Haglund, M.~Haiman, and N.~Loehr.
\newblock A combinatorial formula for nonsymmetric {M}acdonald polynomials.
\newblock {\em Amer. J. Math.}, 130(2):359--383, 2008.

\bibitem[KM17]{KalizeszewskiMorse17}
Ryan Kaliszewski and Jennifer Morse.
\newblock Colorful combinatorics and macdonald polynomials.
\newblock 2017.
\newblock arXiv:1710.00801.

\bibitem[KMO15]{KMO15}
Atsuo Kuniba, Shouya Maruyama, and Masato Okado.
\newblock Multispecies {TASEP} and combinatorial {$R$}.
\newblock {\em J. Phys. A}, 48(34):34FT02, 19, 2015.

\bibitem[Lig75]{Liggett2}
Thomas~M. Liggett.
\newblock Ergodic theorems for the asymmetric simple exclusion process.
\newblock {\em Trans. Amer. Math. Soc.}, 213:237--261, 1975.

\bibitem[Lig05]{Liggett}
Thomas~M. Liggett.
\newblock {\em Interacting particle systems}.
\newblock Classics in Mathematics. Springer-Verlag, Berlin, 2005.
\newblock Reprint of the 1985 original.

\bibitem[Mac95]{Macdonald}
I.~G. Macdonald.
\newblock {\em Symmetric functions and {H}all polynomials}.
\newblock Oxford Mathematical Monographs. The Clarendon Press, Oxford
  University Press, New York, second edition, 1995.
\newblock With contributions by A. Zelevinsky, Oxford Science Publications.

\bibitem[Man17]{Man17}
Olya Mandelshtam.
\newblock Toric tableaux and the inhomogeneous two-species tasep on a ring.
\newblock 2017.
\newblock arXiv:1707.02663.

\bibitem[Mar18]{Martin}
James~B. Martin.
\newblock Stationary distributions of the multi-type {A}{S}{E}{P}s.
\newblock 2018.
\newblock arXiv:1810.10650.

\bibitem[MGP68]{bio}
J~Macdonald, J~Gibbs, and A~Pipkin.
\newblock Kinetics of biopolymerization on nucleic acid templates.
\newblock {\em Biopolymers}, 6, 1968.

\bibitem[Opd95]{Opdam}
Eric~M. Opdam.
\newblock Harmonic analysis for certain representations of graded {H}ecke
  algebras.
\newblock {\em Acta Math.}, 175(1):75--121, 1995.

\bibitem[PEM09a]{ProlhacEvansMallick09}
S.~Prolhac, M.~R. Evans, and K.~Mallick.
\newblock The matrix product solution of the multispecies partially asymmetric
  exclusion process.
\newblock {\em J. Phys. A}, 42(16):165004, 25, 2009.

\bibitem[PEM09b]{Ansatz2}
S.~Prolhac, M.~R. Evans, and K.~Mallick.
\newblock The matrix product solution of the multispecies partially asymmetric
  exclusion process.
\newblock {\em J. Phys. A}, 42(16):165004, 25, 2009.

\bibitem[Spi70]{Spitzer}
Frank Spitzer.
\newblock Interaction of {M}arkov processes.
\newblock {\em Advances in Math.}, 5:246--290 (1970), 1970.

\bibitem[{The}18]{sagemath}
{The Sage Developers}.
\newblock {\em {S}ageMath, the {S}age {M}athematics {S}oftware {S}ystem
  ({V}ersion v8.2)}, 2018.
\newblock {\tt http://www.sagemath.org}.

\bibitem[USW04]{USW}
Masaru Uchiyama, Tomohiro Sasamoto, and Miki Wadati.
\newblock Asymmetric simple exclusion process with open boundaries and
  {A}skey-{W}ilson polynomials.
\newblock {\em J. Phys. A}, 37(18):4985--5002, 2004.

\end{thebibliography}

\end{document}